\documentclass[12pt]{amsart}
\usepackage{amscd,amsmath,amssymb,amsfonts}
\usepackage[cmtip, all]{xy}

\newtheorem{thm}{Theorem}[section]
\newtheorem{prop}[thm]{Proposition}
\newtheorem{lem}[thm]{Lemma}
\newtheorem{cor}[thm]{Corollary}

\theoremstyle{definition}

\newtheorem{defn}[thm]{Definition}
\newtheorem{remk}[thm]{Remark}
\newtheorem{remks}[thm]{Remarks}

\newtheorem{exm}[thm]{Example}
\newtheorem{exms}[thm]{Examples}
\newtheorem{notat}[thm]{Notation}

\numberwithin{equation}{section}

{\hfill$\square$\end{defn}}

{\hfill$\square$\end{remk}}

{\hfill$\square$\end{remks}}

{\hfill$\square$\end{exm}}

{\hfill$\square$\end{exms}}

{\hfill$\square$\end{notat}}

\newcommand{\sC}{{\mathcal C}}

\newcommand{\sO}{{\mathcal O}}

\newcommand{\sR}{{\mathcal R}}

\newcommand{\sV}{{\mathcal V}}

\newcommand{\sZ}{{\mathcal Z}}
\newcommand{\A}{{\mathbb A}}

\newcommand{\G}{{\mathbb G}}

\newcommand{\bL}{{\mathbb L}}

\newcommand{\N}{{\mathbb N}}
\renewcommand{\P}{{\mathbb P}}
\newcommand{\Q}{{\mathbb Q}}

\newcommand{\Z}{{\mathbb Z}}

\newcommand{\surj}{\twoheadrightarrow}
\newcommand{\inj}{\hookrightarrow}

\newcommand{\Pic}{{\rm Pic}}

\newcommand{\Hom}{{\rm Hom}}

\newcommand{\0}{\emptyset}

\renewcommand{\max}{{\operatorname{\rm max}}}

\newcommand{\ds}{{/\kern-3pt/}}

\newcommand{\ov}{\overline}
\newcommand{\wt}{\widetilde}

\newcommand{\tuborg}{\left\{\begin{array}{ll}}
\newcommand{\sluttuborg}{\end{array}\right.}

\begin{document}
\title{Equivariant cobordism for torus actions}
\author{Amalendu Krishna}
\address{School of Mathematics, Tata Institute of Fundamental Research,  
Homi Bhabha Road, Colaba, Mumbai, India} 
\email{amal@math.tifr.res.in}

\baselineskip=10pt 
  

\begin{abstract}
We study the equivariant cobordism groups for the action of a split
torus $T$ on varieties over a field $k$ of characteristic zero. 
We show that for $T$ acting on a variety $X$, there is an isomorphism 
$\Omega^T_*(X) \otimes_{\Omega^*(BT)} \bL \xrightarrow{\cong} \Omega_*(X)$. 
As applications, we show that for a connected linear algebraic group
$G$ acting on a $k$-variety $X$, the forgetful map $\Omega^G_*(X) \to
\Omega_*(X)$ is surjective with rational coefficients.
As a consequence, we describe the rational algebraic cobordism ring of 
algebraic groups and flag varieties.

We prove a structure theorem for the equivariant cobordism of smooth
projective varieties with torus action. Using this, we prove various 
localization theorems and a form of Bott residue formula for such
varieties. As an application, we show that the equivariant cobordism
of a smooth variety $X$ with torus action is generated by the invariant
cobordism cycles in $\Omega_*(X)$ as $\Omega^*(BT)$-module. 
\end{abstract}
\maketitle

\section{Introduction}
Let $k$ be a field of characteristic zero and let $G$ be a linear algebraic
group over $k$. The equivariant algebraic cobordism groups for smooth
varieties were defined by Deshpande in \cite{DD}. They were subsequently
developed into a complete theory of equivariant cobordism for all 
$k$-schemes in \cite{Krishna4}. This theory is based on the analogous 
construction of the equivariant Chow groups by Totaro \cite{Totaro1}
and Edidin-Graham \cite{EG}.
In \cite{Krishna4}, we established all the basic properties
of the equivariant cobordism which are known for the non-equivariant
cobordism theory of Levine and Morel \cite{LM}. 
In this paper, we continue the study of these cobordism groups with 
special focus on the case when the underlying group is a torus.

It was shown in \cite[Theorem~8.7]{Krishna4} that with rational coefficients,
the equivariant cobordism of a $k$-variety $X$ with the action of a connected 
linear algebraic group $G$ is simply the subgroup of invariants inside the 
equivariant cobordism for the action of a maximal torus of $G$ under the 
action of the Weyl group. This reduces most of the computations of
equivariant cobordism to the case when the underlying group is a torus.
Our aim in this paper is to study this special case in more detail and
derive some important consequences for the action of arbitrary connected
groups. These results are applied in \cite{Krishna3} to describe
the non-equivariant cobordism rings of principal and flag bundles.
The results of this paper are also used in \cite{KU} to compute the
algebraic cobordism of toric varieties.  
Some more applications will appear in \cite{KK}. We now describe some
of the main results of this paper.

In this paper, a scheme will mean a quasi-projective $k$-scheme and all 
$G$-actions will be assumed to be linear. 
Let $T$ be a split torus and let $S(T)$ denote the cobordism ring 
$\Omega^*_T(k)$ of the classifying space of $T$. If $X$ is a scheme with
$T$-action, we show in Theorem~\ref{thm:FF*} that the forgetful map
from the equivariant cobordism to the ordinary cobordism group of $X$ 
induces an isomorphism
\[
r^T_X : \Omega^T_*(X) \otimes_{S(T)} \bL \xrightarrow{\cong} \Omega_*(X),
\]
which is a ring isomorphism if $X$ is smooth.
This result for the equivariant Chow groups was earlier proven by Brion
in \cite[Corollary~2.3]{Brion2}.

Using Theorem~\ref{thm:FF*} and the results of Calmes-Petrov-Zainoulline
\cite{CPZ}, we give a explicit geometric description of the algebraic
cobordism ring of flag varieties with rational coefficients.
In particular, we show in Theorem~\ref{thm:flag-V} that if $G$ is a connected 
reductive group and $B$ is a Borel subgroup containing a split maximal torus 
$T$, then with rational coefficients, there is an $\bL$-algebra isomorphism
\begin{equation}\label{eqn:flag-V1}
S(T) {\otimes}_{S(G)} \bL \xrightarrow{\cong} \Omega^*(G/B),
\end{equation}
where $S(G) = \Omega^*(BG)$.
This can be interpreted as the uncompleted version of the results of 
\cite{CPZ}.
In case of $G = GL_n$, this yields an explicit formula for the ring
$\Omega^*(G/B)$ as the quotient of the standard polynomial ring 
$\bL[x_1, \cdots , x_n]$ by the ideal generated by the homogeneous symmetric 
polynomials of strictly positive degree. This latter result for $GL_n$ 
recovers the theorem of Hornbostel and Kiritchenko \cite{HK} by a simpler 
method. We remark that the result of Hornbostel and Kiritchenko is
stronger in the sense that they prove it with the integral coefficients.
We also obtain similar description for $\Omega^*(G)$ that generalizes the 
results of Yagita \cite{Yagita}.

As an application of Theorems~\ref{thm:FF*} and ~\ref{thm:flag-V}, we show 
that if $G$ is a connected linear algebraic group acting on a scheme $X$, then 
the forgetful map
\[
r^G_X : \Omega^G_*(X) \to \Omega_*(X)
\]
is surjective with rational coefficients.
This generalizes the analogous results for the $K$-groups and Chow groups
by Graham \cite{Graham} and Brion \cite{Brion2} to algebraic cobordism.

In our next result, we give a structure theorem 
({\sl cf.} Theorem~\ref{thm:Main-Str}) for the equivariant
cobordism of smooth projective varieties with torus action. The main point
of this theorem is that for such a variety $X$, the equivariant
cobordism of $X$ is very closely related to the non-equivariant cobordism
of the fixed point loci in $X$. The main ingredients in the proof are the
self-intersection formula for the equivariant cobordism in
Proposition~\ref{prop:SIF} and a decomposition theorem of Bialynicki-Birula
for such varieties. 

As an application of Theorem~\ref{thm:Main-Str}, we show that if a split
torus $T$ acts on a smooth variety $X$, then the equivariant cobordism
group $\Omega^T_*(X)$ is generated by the $T$-invariant cobordism cycles
in $\Omega_*(X)$ as $S(T)$-module. For equivariant Chow groups, this
was earlier proven by Brion in \cite[Theorem~2.1]{Brion2}.
The result of Brion can also be deduced from corresponding result for
cobordism and \cite[Proposition~7.1]{Krishna4}.

In \cite{Brion2}, Brion proves the localization formulae for the
equivariant Chow groups for torus action on smooth projective varieties. 
These formulae describe the equivariant Chow groups of a smooth projective
variety $X$ with $T$-action in terms of the equivariant Chow groups of the
fixed point locus $X^T$. Since $\Omega^T_*(X^T)$ is relatively simpler to
compute, these formulae yield a way of computing the equivariant Chow
groups of $X$. In the final set of results in this paper, we prove 
these localization formulae for the equivariant cobordism. Our results
generalize all the analogous results of Brion to the case of
cobordism. These results are expected to be the foundational step in the
computation of the cobordism groups of spherical varieties.
Using the above result, we aim to compute the cobordism ring of certain 
spherical varieties in \cite{KK}. 

\section{Recollection of equivariant cobordism}\label{section:AC}
Since we shall be concerned with the study of 
schemes with group actions and the associated quotient schemes, and since 
such quotients often require the original scheme to be quasi-projective,
we shall assume throughout this paper that all schemes over $k$ are 
quasi-projective. 

In this section, we briefly recall the definition of equivariant
algebraic cobordism and some of its main properties from \cite{Krishna4}.
Since most of the results of \cite{Krishna4} will be repeatedly used in this 
text, we summarize them here for reader's convenience. 
For the definition and all details about the algebraic
cobordism used in this paper, we refer the reader the work of Levine and
Morel \cite{LM}. 
\\
\\
\noindent
{\bf Notations.} We shall denote the category of quasi-projective $k$-schemes 
by $\sV_k$. By a scheme, we shall mean an object 
of $\sV_k$. The category of smooth quasi-projective schemes
will be denoted by $\sV^S_k$. If $G$ is a linear algebraic group over $k$, we 
shall denote the category of quasi-projective $k$-schemes with a $G$-action
and $G$-equivariant maps by $\sV_G$. The associated category of smooth
$G$-schemes will be denoted by $\sV_G^S$. All $G$-actions in this paper will be
assumed to be linear. Recall that this means that all $G$-schemes are
assumed to admit $G$-equivariant ample line bundles. This assumption is
always satisfied for normal schemes ({\sl cf.} \cite[Theorem~2.5]{Sumihiro}, 
\cite[5.7]{Thomason1}).

Recall that the Lazard ring $\bL$ is a polynomial ring over $\Z$ on infinite 
but countably many variables and is given by the quotient of the polynomial 
ring $\Z[A_{ij}| (i,j) \in \N^2]$ by the relations, which uniquely define the 
universal formal group law $F_{\bL}$ of rank one on $\bL$. 
Recall that a cobordism cycle over a $k$-scheme $X$ is
a family $\alpha = [Y \xrightarrow{f} X, L_1, \cdots , L_r]$, where $Y$ is a
smooth scheme, the map $f$ is projective, and $L_i$'s are line bundles on $Y$.
Here, one allows the set of line bundles to be empty.
The degree of such a cobordism cycle is  defined to be ${\rm deg}(\alpha) =
{\rm dim}_k(Y)-r$ and its codimension is defined to be ${\rm dim}(X) -
{\rm deg}(\alpha)$. If $\sZ_*(X)$ is the free abelian group generated by the 
cobordism cycles of the above type with $Y$ irreducible, then
$\sZ_*(X)$ is graded by the degree of cycles.
The algebraic cobordism group of $X$ is defined as
\[
\Omega_*(X) = \frac{\sZ_*(X)}{\sR_*(X)},
\]
where ${\sR_*(X)}$ is the graded subgroup generated by relations which are
determined by the dimension and the section axioms and the above formal group 
law. If $X$ is equi-dimensional, we set $\Omega^i(X) = \Omega_{{\rm dim}(X)-i}(X)$
and grade $\Omega^*(X)$ by codimension of the cobordism cycles. 
It was shown by Levine and Pandharipande \cite{LP} that the cobordism group
$\Omega_*(X)$ can also be defined as the quotient
\[
\Omega_*(X) = \frac{\sZ'_*(X)}{\sR'_*(X)},
\]
where $\sZ'_*(X)$ is the free abelian group on cobordism cycles 
$[Y \xrightarrow{f} X]$ with $Y$ smooth and irreducible and $f$ projective.
The graded subgroup $\sR'_*(X)$ is generated by cycles satisfying the relation
of {\sl double point degeneration}.

Let $X$ be a $k$-scheme of dimension $d$. For $j \in \Z$, let $Z_j$ be the
set of all closed subschemes $Z \subset X$ such that ${\rm dim}_k(Z) \le j$
(we assume ${\rm dim}(\0) = - \infty$). The set $Z_j$ is then ordered by 
the inclusion. For $i \ge 0$,  we set
\[
\Omega_i(Z_j) = {\underset{Z \in Z_j} \varinjlim} \Omega_i(Z) 
\ \ {\rm and} \ \
\Omega_*(Z_j) = {\underset{i \ge 0} \bigoplus} \ \Omega_i(Z_j).
\]
It is immediate that $\Omega_*(Z_j)$ is a graded $\bL_*$-module and there is
a graded $\bL_*$-linear map $\Omega_*(Z_j) \to \Omega_*(X)$.
We define $F_j\Omega_*(X)$ to be the image of the natural $\bL_*$-linear map
$\Omega_*(Z_j) \to \Omega_*(X)$. In other words, $F_j\Omega_*(X)$ is the image
of all $\Omega_*(W) \to \Omega_*(X)$, where $W \to X$ is a projective map
such that ${\rm dim}({\rm Image}(W)) \le j$.
One checks at once that there is a canonical {\sl niveau filtration}
\begin{equation}\label{eqn:niveau1}
0 = F_{-1}\Omega_*(X) \subseteq F_0\Omega_*(X) \subseteq \cdots \subseteq
F_{d-1}\Omega_*(X) \subseteq F_d\Omega_*(X) = \Omega_*(X).
\end{equation}

\subsection{Equivariant cobordism}
In this text, $G$ will denote a linear algebraic group of dimension $g$ 
over $k$. All representations of $G$ will be finite dimensional. 
Recall form \cite{Krishna4} that for any integer $j \ge 0$, a {\sl good pair}  
$\left(V_j,U_j\right)$ corresponding to $j$ for the $G$-action is a pair 
consisting of a $G$-representation $V_j$ and an open subset $U_j \subset V_j$ 
such that the codimension of the complement is at least $j$ and $G$ acts 
freely on $U_j$ with quotient ${U_j}/G$ a quasi-projective scheme.
It is known that such good pairs always exist. 
  
Let $X$ be a $k$-scheme of dimension $d$ with a $G$-action.
For $j \ge 0$, let $(V_j, U_j)$ be an $l$-dimensional good pair 
corresponding to $j$. For $i \in \Z$, if we set
\begin{equation}\label{eqn:E-cob*}
{\Omega^G_i(X)}_j =  \frac{\Omega_{i+l-g}\left({X\stackrel{G} {\times} U_j}\right)}
{F_{d+l-g-j}\Omega_{i+l-g}\left({X\stackrel{G} {\times} U_j}\right)},
\end{equation}
then it is known that ${\Omega^G_i(X)}_j$ is independent of the 
choice of the good pair $(V_j, U_j)$. Moreover, there is a natural surjective 
map $\Omega^G_i(X)_{j'} \surj \Omega^G_i(X)_j$ for $j' \ge j \ge 0$.

\begin{defn}\label{defn:ECob}
Let $X$ be a $k$-scheme of dimension $d$ with a $G$-action. For any 
$i \in \Z$, we define the {\sl equivariant algebraic cobordism} of $X$ to be 
\[
\Omega^G_i(X) = {\underset {j} \varprojlim} \ \Omega^G_i(X)_j.
\]
\end{defn}
The reader should note from the above definition that unlike the ordinary
cobordism, the equivariant algebraic cobordism $\Omega^G_i(X)$ can be 
non-zero for any $i \in \Z$. We set 
\[
\Omega^G_*(X) = {\underset{i \in \Z} \bigoplus} \ \Omega^G_i(X).
\]
If $X$ is an equi-dimensional $k$-scheme with $G$-action, we let
$\Omega^i_G(X) = \Omega^G_{d-i}(X)$ and $\Omega^*_G(X) =
{\underset{i \in \Z} \oplus} \ \Omega^i_G(X)$. 
It is known that if $G$ is trivial, then the $G$-equivariant cobordism
reduces to the ordinary one.
\begin{remk}\label{remk:Csmooth}
It is easy to check from the above definition of the niveau filtration 
that if $X$ is a smooth and irreducible $k$-scheme of dimension $d$,
then $F_j\Omega_i(X) = F^{d-j}\Omega^{d-i}(X)$, where $F^{\bullet}\Omega^*(X)$ 
is the coniveau filtration used in \cite{DD}. Furthermore, one also 
checks in this case that if $G$ acts on $X$, then
\begin{equation}\label{eqn:Csmooth1}
\Omega^i_G(X) = {\underset {j} \varprojlim} \
\frac{\Omega^{i}\left({X\stackrel{G} {\times} U_j}\right)}
{F^{j}\Omega^{i}\left({X\stackrel{G} {\times} U_j}\right)},
\end{equation}
where $(V_j, U_j)$ is a good pair corresponding to any $j \ge 0$.
Thus the above definition of the equivariant cobordism
coincides with that of \cite{DD} for smooth schemes.
\end{remk}
For equi-dimensional schemes, we shall write the (equivariant) cobordism groups
cohomologically.
The $G$-equivariant cobordism group $\Omega^*(k)$ of the ground field $k$ 
is denoted by $\Omega^*(BG)$ and is called the cobordism of the 
{\sl classifying space} of $G$. We shall often write it as $S(G)$.

The following important result shows that if 
we suitably choose a sequence of good pairs ${\{(V_j, U_j)\}}_{j \ge 0}$, then 
the above equivariant cobordism group can be computed without taking 
quotients by the niveau filtration. This is often very helpful in
computing the equivariant cobordism groups. 
  
\begin{thm}\label{thm:NO-Niveu}$(${\sl cf.} \cite[Theorem~6.1]{Krishna4}$)$
Let ${\{(V_j, U_j)\}}_{j \ge 0}$ be a sequence of $l_j$-dimensional 
good pairs such that \\
$(i)$ $V_{j+1} = V_j \oplus W_j$ as representations of $G$ with ${\rm dim}(W_j) 
> 0$ and \\
$(ii)$ $U_j \oplus W_j \subset U_{j+1}$ as $G$-invariant open subsets. \\ 
Then for any scheme $X$ as above and any $i \in \Z$,
\[
\Omega^G_i(X) \xrightarrow{\cong} {\underset{j}\varprojlim} \
\Omega_{i+l_j-g}\left(X \stackrel{G}{\times} U_j\right).
\]
Moreover, such a sequence ${\{(V_j, U_j)\}}_{j \ge 0}$ of good pairs always 
exists.
\end{thm}  

\subsection{Change of groups}
If $H \subset G$ is a closed subgroup of dimension $h$, then any 
$l_j$-dimensional good pair 
$(V_j, U_j)$ for $G$-action is also a good pair for the induced $H$-action. 
Moreover, for any $X \in \sV_G$ of dimension $d$, 
$X \stackrel{H}{\times} U_j \to X \stackrel{G}{\times} U_j$
is an \'etale locally trivial $G/H$-fibration and hence a smooth map 
({\sl cf.} \cite[Theorem~6.8]{Borel}) of relative dimension $g-h$. 
Taking the inverse limit of corresponding pull-back maps on the
cobordism groups, this induces the restriction map
\begin{equation}\label{eqn:res}
r^G_{H,X} : \Omega^G_*(X) \to \Omega^H_*(X).
\end{equation}
Taking $H = \{1\}$, we get the {\sl forgetful} map 
\begin{equation}\label{eqn:res}
r^G_X : \Omega^G_*(X) \to \Omega_*(X)
\end{equation}
from the equivariant to the non-equivariant cobordism. 
Since $r^G_{H,X}$ is obtained as a pull-back under the smooth map, it commutes
with any projective push-forward and smooth pull-back ({\sl cf.} 
Theorem~\ref{thm:Basic}).

The equivariant cobordism for the action of a group $G$ is related with
the equivariant cobordism for the action of the various subgroups of $G$
by the following.
\begin{prop}[Morita Isomorphism]\label{prop:Morita}
Let $H \subset G$ be a closed subgroup and let $X \in {\sV}_H$.
Then there is a canonical isomorphism
\begin{equation}\label{eqn:MoritaI}
\Omega^G_*\left(G \stackrel{H}{\times} X\right) \xrightarrow{\cong}
\Omega^H_*(X).
\end{equation}
\end{prop}

\subsection{Fundamental class of cobordism cycles}\label{subsection:FCL}
Let $X \in \sV_G$ and let $Y \xrightarrow{f} X$ be a morphism in $\sV_G$
such that $Y$ is smooth of dimension $d$ and $f$ is projective. For any
$j \ge 0$ and any $l$-dimensional good pair $(V_j, U_j)$, 
$[Y_G \xrightarrow{f_G} X_G]$ is an ordinary cobordism cycle of dimension
$d+l-g$ by [{\sl loc.cit.}, Lemma~5.3] and hence defines an element 
$\alpha_j \in {\Omega^G_d(X)}_j$. Moreover, it is evident that the image of 
$\alpha_{j'}$ is $\alpha_j$ for $j' \ge j$. Hence we get a unique element
$\alpha \in \Omega^G_d(X)$, called the {\sl G-equivariant fundamental class}
of the cobordism cycle $[Y \xrightarrow{f} X]$. We also see from this
more generally that if $[Y \xrightarrow{f} X, L_1, \cdots, L_r]$ is as
above with each $L_i$ a $G$-equivariant line bundle on $Y$, then this defines
a unique class in $\Omega^G_{d-r}(X)$. It is interesting question to ask under 
what conditions on the group $G$, the equivariant cobordism group 
$\Omega^G_*(X)$ is generated by the fundamental classes of $G$-equivariant
cobordism cycles on $X$. We shall show in this text that this is indeed true 
for a torus action on smooth varieties.  

\subsection{Basic properties}
The following result summarizes the basic properties of the equivariant
cobordism.

\begin{thm}\label{thm:Basic}$(${\sl cf.} \cite[Theorems~5.1, 5.4]{Krishna4}$)$
The equivariant algebraic cobordism satisfies the following properties. \\
$(i)$ {\sl Functoriality :} The assignment $X \mapsto \Omega_*(X)$ is
covariant for projective maps and contravariant for smooth maps in
$\sV_G$. It is also contravariant for l.c.i. morphisms in $\sV_G$. 
Moreover, for a fiber diagram 
\[
\xymatrix@C.7pc{
X' \ar[r]^{g'} \ar[d]_{f'} & X \ar[d]^{f} \\
Y' \ar[r]_{g} & Y}
\]
in $\sV_G$ with $f$ projective and $g$ smooth, one has 
$g^* \circ f_* = {f'}_* \circ {g'}^* : \Omega^G_*(X) \to \Omega^G_*(Y')$.
\\
$(ii) \ Localization :$ For a $G$-scheme $X$ and a closed $G$-invariant
subscheme $Z \subset X$ with complement $U$, 
there is an exact sequence
\[
\Omega^G_*(Z) \to \Omega^G_*(X) \to \Omega^G_*(U) \to 0.
\]
$(iii) \ Homotopy :$  If $f : E \to X$ is a $G$-equivariant vector bundle,
then $f^*: \Omega^G_*(X) \xrightarrow{\cong} \Omega^G_*(E)$. \\
$(iv) \ Chern \ classes :$ For any $G$-equivariant vector bundle $E
\xrightarrow{f} X$ of rank $r$, there are equivariant Chern class operators
$c^G_m(E) : \Omega^G_*(X) \to \Omega^G_{*-m}(X)$ for $0 \le m \le r$ with
$c^G_0(E) = 1$. These Chern classes 
have same functoriality properties as in the non-equivariant case.
Moreover, they satisfy the Whitney sum formula. \\
$(v) \ Free \ action :$ If $G$ acts freely on $X$ with quotient $Y$, then
$\Omega^G_*(X) \xrightarrow{\cong} \Omega_*(Y)$. \\
$(vi) \ Exterior \ Product :$ There is a natural product map
\[
\Omega^G_i(X) \otimes_{\Z} \Omega^G_{i'}(X') \to \Omega^G_{i+i'}(X \times X').
\]
In particular, $\Omega^G_*(k)$ is a graded algebra and $\Omega^G_*(X)$ is
a graded $\Omega^G_*(k)$-module for every $X \in \sV_G$. 
For $X$ smooth, the pull-back via the diagonal $X \inj X \times X$ turns
$\Omega^*_G(X)$ into an $S(G)$-algebra. \\
$(vii) \ Projection \ formula :$ For a projective map $f : X' \to X$ in
$\sV^S_G$, one has for $x \in \Omega^G_*(X)$ and $x' \in \Omega^G_*(X')$,
the formula : $f_*\left(x' \cdot f^*(x)\right) = f_*(x') \cdot x$. \\
\end{thm}

\subsection{Formal group law}\label{subsection:FGL*}
We recall from \cite{Krishna4} that the first Chern class of the tensor 
product of two equivariant line bundles satisfies the formal group law of the 
ordinary cobordism. That is, for $L_1, L_2 \in {\rm Pic}^G(X)$, one has
\begin{equation}\label{eqn:FGL}
c^G_1(L_1 \otimes L_2) = c^G_1(L_1) + c^G_1(L_2) + c^G_1(L_1) c^G_1(L_2)
{\underset{i,j \ge 1}\sum} \ a_{i,j}  \left(c^G_1(L_1)\right)^{i-1}
\left(c^G_1(L_2)\right)^{j-1},
\end{equation}
where $F(u, v) = u+v + uv{\underset{i,j \ge 1}\sum} \ a_{i,j}u^{i-1}v^{j-1},
\ a_{i,j} \in \bL_{1-i-j}$ is the universal formal group law on $\bL$.
We shall often write $c^G_1(L_1 \otimes L_2) = c^G_1(L_1) {+}_{F} c^G_1(L_2)$.

If $X$ is smooth, the commutative sub-$\bL$-algebra (under composition) of 
${\rm End}_{\bL}\left(\Omega^*_G(X)\right)$ generated
by the Chern classes of the vector bundles is canonically identified 
with a sub-$\bL$-algebra of the cobordism ring $\Omega^*_G(X)$ via the
identification $c^G_i(E) \mapsto \wt{c^G_i(E)} =
c^G_i(E)\left([X \xrightarrow{id} X]\right)$.
We shall denote this image also by $c^G_i(E)$ or, by $c^G_i$ if the underlying 
vector bundle is understood. The formal group law 
of the algebraic cobordism then gives a map of pointed sets
\begin{equation}\label{eqn:Chern-map}
\Pic^G(X) \to \Omega^1_G(X), \ \ L \mapsto c^G_1(L)
\end{equation}
such that $c^G_1(L_1 \otimes L_1) = c^G_1(L_1) +_F c^G_1(L_2)$.
It is known that even though $c^G_1(L)$ is not nilpotent in $\Omega^*_G(X)$
for $L \in {\rm Pic}^G(X)$ (unlike in the ordinary case), 
$c^G_1(L_1) +_F c^G_1(L_2)$ is a well defined element of $\Omega^1_G(X)$. 
In this paper, we shall view the (equivariant) Chern classes as elements of 
the (equivariant) cobordism ring of a smooth variety in the above sense.
In the rest of this paper, the sum $x +_F y$ for $x,y \in \Omega^1_G(X)$ will 
denote the element $F(x,y) \in \Omega^1_G(X)$, the addition according to the 
formal group law.

\subsection{Cobordism ring of classifying spaces}\label{subsection:CCS}
Let $R$ be a Noetherian ring and let $A = {\underset{j \in \Z}\oplus} A_j$ be 
a $\Z$-graded $R$-algebra with $R \subset A_0$. Recall that the {\sl 
graded power series ring} $S^{(n)} = {\underset{i \in \Z} \oplus} S_i$ is a
graded ring such that $S_i$ is the set of formal power series of the form 
$f({\bf t}) = {\underset{m({\bf t}) \in \sC} \sum} a_{m({\bf t})} m({\bf t})$
such that $a_{m({\bf t})}$ is a homogeneous element of $A$ of degree 
$|a_{m({\bf t})}| $ and $|a_{m({\bf t})}| + |m({\bf t})| = i$. 
Here, $\sC$ is the set of all monomials in 
${\bf t} = (t_1, \cdots , t_n)$ and $|m({\bf t})| = i_1 + \cdots + i_n$ if 
$m({\bf t}) = t^{i_1}_1 \cdots t^{i_n}_n$.
We call $|m({\bf t})|$ to be the degree of the monomial $m({\bf t})$.

We shall often write the above graded power series ring as
${A[[{\bf t}]]}_{\rm gr}$ to distinguish it from the
usual formal power series ring. 
Notice that if $A$ is only non-negatively graded, then $S^{(n)}$ is nothing but 
the standard polynomial ring $A[t_1, \cdots , t_n]$ over $A$. It is also easy 
to see that $S^{(n)}$ is indeed a graded ring which is a subring of the formal 
power series ring $A[[t_1, \cdots , t_n]]$. The following result summarizes 
some basic properties of these rings. The proof is straightforward and is left 
as an exercise.

\begin{lem}\label{lem:GPSR}
$(i)$ There are inclusions of rings $A[t_1, \cdots , t_n] \subset S^{(n)} \subset
A[[t_1, \cdots , t_n]]$, where the first is an inclusion of graded rings. \\
$(ii)$ These inclusions are analytic isomorphisms with respect to the
${\bf t}$-adic topology. In particular, the induced maps of the associated
graded rings
\[
A[t_1, \cdots , t_n] \to {\rm Gr}_{({\bf t})} S^n \to 
{\rm Gr}_{({\bf t})} A[[t_1, \cdots , t_n]]
\]
are isomorphisms. \\
$(iii)$ $S^{(n-1)}[[t_i]]_{\rm gr} \xrightarrow{\cong} S^{(n)}$. \\
$(iv)$ $\frac{S^{(n)}}{(t_{i_1}, \cdots , t_{i_r})} \xrightarrow{\cong} S^{(n-r)}$ 
for any $n \ge r \ge 1$, where $S^{(0)} = A$. \\ 
$(v)$ The sequence $\{t_1, \cdots , t_n\}$ is a regular sequence in $S^{(n)}$.
\\
$(vi)$ If $A = R[x_1, x_2, \cdots ]$ is a polynomial ring with
$|x_i| < 0$ and ${\underset{i \to \infty}{\rm lim}} \ |x_i| = - \infty$, then
$S^{(n)} \xrightarrow{\cong}
{\underset{i}\varprojlim} \ {R[x_1, \cdots , x_i][[{\bf t}]]}_{\rm gr}$.
\end{lem}
 
Since we shall mostly be dealing with the graded power series ring in this 
text, we make the convention of writing 
${A[[{\bf t}]]}_{\rm gr}$ as 
$A[[{\bf t}]]$, while the standard formal
power series ring will be written as $\widehat{A[[{\bf t}]]}$. 

It is known \cite[Proposition~6.5]{Krishna4} that if $T$ is a split torus of 
rank $n$ and if $\{\chi_1, \cdots , \chi_n \}$ is a chosen basis 
of  the character group $\widehat{T}$, then there is a canonical isomorphism of
graded rings 
\begin{equation}\label{eqn:CBT*}
\bL[[t_1, \cdots , t_n]] \xrightarrow{\cong} \Omega^*(BT), \ \ 
t_i \mapsto c^T_1(L_{\chi_i}).
\end{equation}
Here, $L_{\chi}$ is the $T$-equivariant line bundle on ${\rm Spec}(k)$
corresponding to the character $\chi$ of $T$.
One also has isomorphisms 
\begin{equation}\label{eqn:BT2}
\Omega^*(BGL_n) \xrightarrow{\cong} \bL[[\gamma_1, \cdots , \gamma_n]] \ \ 
{\rm and} \ \
\Omega^*(BSL_n) \xrightarrow{\cong} \bL[[\gamma_2, \cdots , \gamma_n]]
\end{equation}
of graded $\bL$-algebras, where $\gamma_i$'s  are the elementary symmetric
polynomials in $t_1, \cdots , t_n$ that occur in $\Omega^*(BT)$.

We finally recall the following result of \cite{Krishna4} that will be useful
for us.

\begin{thm}\label{thm:W-inv}$(${\sl cf.} \cite[Theorem~8.7]{Krishna4}$)$
Let $G$ be a connected linear algebraic group and let $L$ be a Levi subgroup
of $G$ with a split maximal torus $T$. Let $W$ denote the Weyl group of $L$
with respect to $T$.
Then for any $X \in \sV_G$, the natural map 
\begin{equation}\label{eqn:W-inv1}
\Omega^G_*(X) \to {\left(\Omega^T_*(X)\right)}^W
\end{equation}
is an isomorphism.
\end{thm}

\section{The forgetful map}\label{section:FF-map}
In this section, we study the forgetful map 
\[
r^G_X : \Omega^G_*(X) \to \Omega_*(X)
\]
of ~\eqref{eqn:res} from the
equivariant to the non-equivariant cobordism when $G$ is a split torus. 
It was shown by Brion in \cite[Corollary~2.3]{Brion2} (see also 
\cite[Corollary~1.4]{Krishna1}) that the natural map 
$CH^G_*(X){\otimes}_{S(G)} \Z \to CH_*(X)$ is an isomorphism. Our aim in this
section is to prove an analogous result for the algebraic cobordism.
We do this by using a technique which appears to be simpler than the one 
Brion used for studying the Chow groups. 
For the rest of this paper, we shall denote the cobordism ring
$S(T) = \Omega^*_T(k)$ of the classifying space of a split torus $T$
simply by $S$. Let $I_T \subset S$ be the augmentation ideal so that
$S/{I_T} \xrightarrow{\cong} \bL$. We first prove the following very useful
self-intersection formula for the equivariant cobordism.
\begin{prop}[Self-intersection formula]\label{prop:SIF}
Let $G$ be a linear algebraic group and let $Y \xrightarrow{f} X$ be a regular 
$G$-equivariant embedding in $\sV_G$ of 
pure codimension $d$ and let $N_{Y/X}$ denote the equivariant normal bundle of 
$Y$ inside $X$. Then one has for every $y \in \Omega^G_*(Y)$,
$f^* \circ f_*(y) = c^G_d(N_{Y/X}) (y)$.
\end{prop} 
\begin{proof} First we prove this for the non-equivariant algebraic cobordism.
But this is a direct consequence of the construction of the refined pull-back
map and the excess intersection formula for algebraic cobordism in 
\cite[Section~6]{LM}.
One simply has to use Lemma~6.6.2 and Theorem~6.6.9 of {\sl loc. cit.} and 
follow exactly the same argument as in the proof of the self-intersection
formula for Chow groups in \cite[Theorem~6.2]{Fulton}.

To prove the equivariant version, let $j \ge 0$ and let $(V_j, U_j)$ be a good
pair. Then $Y_G {\overset{\ov{f}}\inj} X_G$ is a regular 
closed embedding. Writing the normal bundle $N_{Y/X}$ simply by $N$, we see
that $N_G$ is the normal bundle of $Y_G$ inside $X_G$ and $c^G_d(N)$ is induced
by $c_d(N_G)$. The non-equivariant self-intersection formula yields
${\ov{f}}^* \circ {\ov{f}}_* = c_d(N_G) (-)$ on $\Omega_*(Y_G)$. 
Moreover, the maps ${\ov{f}}^*$ and 
${\ov{f}}_*$ descend to a compatible system of maps
({\sl cf.} \cite[Theorem~5.4]{Krishna4}) 
\[
{\Omega^G_{*}(Y)}_{j} \xrightarrow{{\ov{f}}_*} {\Omega^G_*(X)}_{j+d}, \
{\Omega^G_{*}(X)}_{j+d} \xrightarrow{{\ov{f}}^*} {\Omega^G_{*-d}(Y)}_{j+d},
\]
whose composite with the natural surjection ${\Omega^G_{*-d}(Y)}_{j+d} \surj
{\Omega^G_{*-d}(Y)}_{j}$ ({\sl cf.} \cite[Lemma~4.3]{Krishna4}) gives a map of 
inverse systems $\{{\Omega^G_{*}(Y)}_{j}\}
\xrightarrow{{\ov{f}}^* \circ {\ov{f}}_*} \{{\Omega^G_{*-d}(Y)}_j\}$. 

The equivariant Chern class $c^G_d(N)$ is induced by taking the inverse limit 
of maps $c^G_{d,j}(N_G) : \{{\Omega^G_{*}(Y)}_{j}\} \to \{{\Omega^G_{*-d}(Y)}_j\}$.
Since ${{\ov{f}}^* \circ {\ov{f}}_*} = c^G_{d,j}(N_G) (-)$ for each 
$j \ge 0$ by the non-equivariant case proven above, we conclude that
$f^* \circ f_* = c^G_d(N) (-)$ on the equivariant cobordism.
\end{proof}
\begin{cor}\label{cor:line-bundle}
Let $G$ be a linear algebraic group acting on a $k$-variety $X$ and let 
$p : L \to X$ be a $G$-equivariant line bundle. Let $f : X \to L$ be the
zero-section embedding. Then $f^* \circ p^* = {\rm Id}$ and $f^*\circ f_* =
c^G_1(L)$. 
\end{cor}
\begin{proof} Since $p$ is an equivariant line bundle, the homotopy invariance
property of Theorem~\ref{thm:Basic} implies the first isomorphism.
The second isomorphism is a direct consequence of Proposition~\ref{prop:SIF}.
\end{proof}
\begin{lem}\label{lem:surj*}
Let $G$ be a connected linear algebraic group acting on a $k$-variety $X$.
Then all irreducible components of $X$ are $G$-invariant and the map
\[
{\underset{1 \le i \le r}\bigoplus} \Omega^G_*(X_i) \to
\Omega^G_*(X)
\]
is surjective, where $\{X_1, \cdots, X_r\}$ is the set of irreducible
components of $X$.
\end{lem}
\begin{proof}
Since $G$ is connected and hence irreducible, it is clear that all components 
of $X$ are $G$-invariant. We now prove the surjectivity assertion.
By an induction on the number of irreducible components, it suffices to
consider the case when
$X = X_1 \cup X_2$, where each $X_i$ is a $G$-invariant closed 
(not necessarily irreducible) subscheme. Set $Y = X_1 \cap X_2$.
It is then clear that $Y$ is a $G$-invariant closed subscheme of $X$.
By Theorem~\ref{thm:Basic} $(ii)$, we get a commutative diagram
\begin{equation}\label{eqn:FORG}
\xymatrix@C.7pc{
\Omega^G_*(Y) \ar[r] \ar[d] & \Omega^G_*(X_1)  \ar[r] \ar[d] & 
\Omega^G_*(X_1 - Y)  \ar[r] \ar@{=}[d] &  0 \\
\Omega^G_*(X_2) \ar[r] & \Omega^G_*(X) \ar[r] & \Omega^G_*(X - X_2) \ar[r] &
0}
\end{equation}
with exact rows.
It follows from the diagram chase that the map $\Omega^G_*(X_1) \oplus
\Omega^G_*(X_2) \to \Omega^G_*(X)$ is surjective.
\end{proof}

\begin{thm}\label{thm:FF*}
Let $T$ be a split torus acting on a $k$-variety $X$. Then the forgetful
map $r^T_X$ induces an isomorphism
\[
\ov{r}^T_X : \Omega^T_*(X) \otimes_{S} \bL \xrightarrow{\cong} \Omega_*(X).
\]
If $X$ is smooth, this is an $\bL$-algebra isomorphism.
\end{thm}
\begin{proof}
We have already seen that $r^T_X$ is an $\bL$-algebra homomorphism if $X$ is
smooth. So we only need to show that it descends to the desired isomorphism
for any $X \in \sV_T$.

Let $n$ be the rank of $T$ and let $\{\chi_1, \cdots , \chi_n\}$ be a chosen
basis of the character group $M$ of $T$. We denote the coordinates of the affine
space $\A^n_k$ by $x = (x_1, \cdots , x_n)$ and define a linear action of $T$ on
$\A^n$ by 
\begin{equation}\label{eqn:action}
t \cdot x = y = (y_1, \cdots , y_n), \ {\rm where} \ y_i = \chi_i(t)x_i \
\forall \ i.
\end{equation}
Let $H_i \subset \A^n$ be the hyperplane $\{x_i = 0\}$ for $1 \le i \le n$.
Then each $H_i$ is $T$-invariant and 
\begin{equation}\label{eqn:action1}
T \cong \A^n - \left(\stackrel{n}{\underset{i=1}\cup} H_i\right) = \A^n - H
\ ({\rm say}).
\end{equation}
Note that since $\{\chi_1, \cdots , \chi_n\}$ is a basis and since 
the action of $\chi_i(t)$ is multiplication by $t_i$ if $t =(t_1, \cdots , t_n)$
is a coordinate of $T$, we see that under the identification in 
~\eqref{eqn:action1}, the restriction of the above action of $T$ on $\A^n$
to the open subset $T$ is by simply the left (and hence right) 
multiplication.

The localization sequence of Theorem~\ref{thm:Basic} $(ii)$ gives an exact 
sequence
\begin{equation}\label{eqn:action2}
\Omega^T_{*}(X \times H) \to \Omega^T_*(X \times \A^n) \to
\Omega^T_*(X \times T) \to 0,
\end{equation}
where $T$ acts on $X \times \A^n$ via the diagonal action. Using 
Proposition~\ref{prop:Morita} and Lemma~\ref{lem:surj*}, this reduces
to an exact sequence 
\begin{equation}\label{eqn:action3}
{\underset{1 \le i \le n}\bigoplus} \Omega^T_{*}(X \times H_i) 
\xrightarrow{\sum j^i_*}
\Omega^T_*(X \times \A^n) \to \Omega_*(X) \to 0,
\end{equation}
where $j^i: H_i \to \A^n$ is the inclusion.
By the equivariant homotopy invariance, we can identify $\Omega^T_*(X \times
H_i)$ by $\Omega^T_*(X)$ via the projection 
\begin{equation}\label{eqn:action4}
\xymatrix{ 
H_i \ar[r]^{j^i} \ar[dr]_{p_i} & \A^n \ar[d]^{p} \\
& {\rm Spec}(k)}
\end{equation}
and similarly identify $\Omega^T_*(X \times \A^n)$ by $\Omega^T_*(X)$ via
$p^*$. Since $T$ acts on $\A^n$ coordinate-wise, we can also equivariantly
identify $\A^n$ as $H_i \times \A^1$ and then by the 
homotopy invariance, we have isomorphisms
$\Omega^T_*(X \times \A^n) \cong \Omega^T_*(X \times \A^1 \times H_i)
\cong  \Omega^T_*(X \times \A^1)$. 

Under the identifications
$\Omega^T_*(X) \stackrel{\cong}{\underset{p^*_i}\to}
\Omega^T_*(X \times H_i)$ and $\Omega^T_*(X \times \A^1) \cong
 \Omega^T_*(X \times \A^n)$,
the map $\Omega^T_{*}(X \times H_i) \xrightarrow{j^i_*}
\Omega^T_*(X \times \A^n)$ is the composite
\begin{equation}\label{eqn:action5} 
\Omega^T_{*}(X) \xrightarrow{j_*} \Omega^T_*(X \times \A^1)
\xrightarrow{j^*} \Omega^T_{*}(X),
\end{equation}
where $j : {\rm Spec}(k) \to \A^1$ is the zero-section embedding. 
In particular, we can apply Corollary~\ref{cor:line-bundle} to
identify the map $j^i_*$ as the map $\Omega^T_{*}(X) 
\xrightarrow{c^T_1(L_{\chi_i})} \Omega^T_{*}(X)$.  
Hence, the first map of ~\eqref{eqn:action3} is identified as the map
\begin{equation}\label{eqn:action6}
{\left(\Omega^T_{*}(X)\right)}^{\oplus n} 
\xrightarrow{\sum c^T_1(L_{\chi_i})} \Omega^T_*(X)
\end{equation}
\[
(a_1, \cdots , a_n) \mapsto \stackrel{n}{\underset{i=1}\sum}
c^T_1(L_{\chi_i}) \cdot a_i.
\]
Since each $L_{\chi_i}$ is a $T$-equivariant line bundle on ${\rm Spec}(k)$,
the above map is same as 
\[
S^{\oplus n} \otimes_S \Omega^T_*(X) \xrightarrow{\Phi \otimes id} S \otimes_S 
\Omega^T_*(X),
\]
where $S = \Omega^*_T(k) = \bL[[x_1, \cdots , x_n]]$ and $\Phi :S^{\oplus n}
\to S$ is given as in ~\eqref{eqn:action6}. 

Now, the map $c^T_1(L_{\chi_i})$ on $S$
is given by the multiplication by $x_i$ for $1 \le i \le n$. In particular,
we see that $[S^{\oplus n} \xrightarrow{\Phi} S]$ is the complex 
$(K_1 \xrightarrow{\Phi} K_0)$, where 
\[
K_{\bullet} = K_{\bullet}(S \xrightarrow{x_i}
S) = (K_n \to \cdots \to K_1 \xrightarrow{\Phi} K_0)
\]
is the Koszul complex associated to the sequence $(x_1, \cdots , x_n)$ in $S$. 
Hence, by comparing this with
~\eqref{eqn:action3} and ~\eqref{eqn:action6}, we conclude that 
~\eqref{eqn:action3} is the exact sequence 
\begin{equation}\label{eqn:action7}
K_1 \otimes_S \Omega^T_*(X) \to K_0 \otimes_S \Omega^T_*(X) \to \Omega_*(X)
\to 0.
\end{equation}
On the other hand, $(x_1, \cdots , x_n)$ is a regular sequence in $S$ by
Lemma~\ref{lem:GPSR}. This in turn implies using 
\cite[Theorem~16.5]{Matsumura} that $K^{\bullet}$ is a resolution of $\bL$ as
$S$-module. In particular, the sequence
\[
K_1 \otimes_S \Omega^T_*(X) \to K_0 \otimes_S \Omega^T_*(X) \to 
\bL \otimes_S \Omega^T_*(X) \to 0
\]
is exact. The proof of the theorem follows by comparing this exact sequence
with ~\eqref{eqn:action7}.
\end{proof}

\begin{remk}\label{remk:GFF*}
We expect the assertion of Theorem~\ref{thm:FF*} to hold for all
linear algebraic groups with rational coefficients. 
Although, we are unable to prove this generalization,
we shall show a weaker version later in this paper.
We shall also show the isomorphism $\Omega^G_*(X) \otimes_{S(G)} \bL
\xrightarrow{\cong} \Omega_*(X)$ when $X$ is the flag variety $G/B$
({\sl cf.} Theorem~\ref{thm:flag-V}). 
\end{remk}

\section{Structure theorems for torus action}
\label{section:Local-torus}
In this section, we prove some structure theorems for the equivariant 
cobordism of smooth projective varieties with torus action.
To state our first result of this kind, recall that for homomorphisms 
$A_i \xrightarrow{f_i} B, \ i= 1, 2$ of abelian groups, 
$A_1 {\underset{B}\times} A_2$ denotes 
the fiber product $\{(a_1, a_2)| f_1(a_1) = f_2(a_2)\}$.
\begin{prop}\label{prop:split-cob}
Let $T$ be a split torus and let $Y \inj X$ be a $T$-equivariant closed 
embedding of smooth varieties of codimension $d \ge 0$. Assume that 
$c^G_d\left(N_{Y/X}\right)$ is a non-zero divisor in the cobordism ring 
$\Omega^*_T\left(Y\right)$. Let $Y \stackrel{i}{\inj} X$ and
$U \stackrel{j}{\inj} X$ be the inclusion maps, where $U$ is the complement
of $Y$ in $X$. Then: \\
$(i)$ The localization sequence 
\[
0 \to \Omega^*_T(Y) \xrightarrow{i_*} \Omega^*_T(X) \xrightarrow{j^*}
\Omega^*_T(U) \to 0
\]
is exact. \\
$(ii)$ The restriction ring homomorphisms
\[
\Omega^*_T(X) \stackrel{(i^*, j^*)}{\longrightarrow} 
\Omega^*_T(Y) \times \Omega^*_T(U)
\]
give an isomorphism of rings
\[
\Omega^*_T(X) \xrightarrow{\cong} \Omega^*_T(Y) {\underset {\wt{\Omega^*_T(Y)}}
{\times}} \Omega^*_T(U),
\]
where $\wt{\Omega^*_T(Y)} = 
{\Omega^*_T(Y)}/{\left(c^T_d\left(N_{Y/X}\right)\right)}$,
and the maps 
\[
\Omega^*_T(Y) \to \wt{\Omega^*_T(Y)}, \ \Omega^*_T(U) \to \wt{\Omega^*_T(Y)}
\]
are respectively, the natural surjection and the map 
\[
\Omega^*_T(U) = {\frac{\Omega^*_T(X)}{i_*\left(\Omega^*_T(Y)\right)}}
\xrightarrow{i^*} {\frac{\Omega^*_T(Y)}{c^T_d\left(N_{Y/X}\right)}}  
= \wt{\Omega^*_T(Y)},
\]
which is well-defined by Proposition~\ref{prop:SIF}.
\end{prop}
\begin{proof} Since all the statements are obvious for $d= 0$, we assume
that $d \ge 1$.
In view of Theorem~\ref{thm:Basic} $(ii)$, we only need to show that $i_*$
is injective to prove the first assertion. But this follows from 
Proposition~\ref{prop:SIF} and the assumption that $c^T_d(N_{Y/X})$
is a non-zero divisor in the ring $\Omega^*_T(Y)$. Since $i^*$ and 
$j^*$ are ring homomorphisms, the proof of the second part follows
directly from the first part and \cite[Lemma~4.4]{VV}.
\end{proof} 

\subsection{The motivic cobordism theory}
Before we prove our other results of this section, we recall the theory
of motivic algebraic cobordism $MGL_{*,*}$ introduced by Voevodsky in
\cite{Voevodsky}. This is a bi-graded ring cohomology theory in the category
of smooth schemes over $k$. Levine has recently shown in \cite{Levine1}
that $MGL_{*,*}$ extends uniquely to a bi-graded oriented Borel-Moore homology
theory $MGL'_{*,*}$ on the category of all schemes over $k$. This homology
theory has exterior products, homotopy invariance, localization exact sequence 
and Mayer-Vietoris among other properties ({\sl cf.} [{\sl loc. cit.}, 
Section~3]). Moreover, the universality of Levine-Morel cobordism theory 
implies that there is a unique map 
\[
\vartheta : \Omega_* \to MGL'_{2*, *}
\]
of oriented Borel-Moore homology theories. Our motivation for studying the
motivic cobordism theory in this text comes from the following result of
Levine.
\begin{thm}[\cite{Levine2}]\label{thm:LCOMP}
For any $X \in \sV_k$, the map $\vartheta_X$ is an isomorphism.
\end{thm}

\begin{prop}\label{prop:filter-Gen}
Let $X$ be a $k$-scheme with a filtration by closed subschemes  
\begin{equation}\label{eqn:filtration-Gen}
{\emptyset} = X_{-1} \subset X_0 \subset \cdots \subset X_n = X
\end{equation}
and maps ${\phi}_m : U_m = (X_m \setminus X_{m-1}) \to Z_m$ for $0 \le m \le n$ 
which are all vector bundles. Assume moreover that each $Z_m$ is smooth and 
projective. Let $MGL'$ be the Motivic Borel-Moore cobordism theory. 
Then there is a canonical isomorphism
\[
\stackrel{n}{\underset{m=0}{\bigoplus}} MGL'_{*,*}\left(Z_m\right)
\xrightarrow{\cong} MGL'_{*,*}\left(X\right).
\]
\end{prop}  
\begin{proof}
We prove it by induction on $n$.
For $n = 0$, the map $X = X_0 \xrightarrow{{\phi}_0} Z_0$ is a  
vector bundle over a smooth scheme and hence the lemma follows from the
homotopy invariance of the $MGL$-theory. 
We now assume by induction that $1 \le m \le n$ and 
\begin{equation}\label{eqn:split0}
\stackrel{m-1}{\underset{j=0}{\bigoplus}} MGL'_{*,*}\left(Z_j\right)
\xrightarrow{\cong} MGL'_{*,*}\left(X_{m-1}\right).
\end{equation}
Let $i_{m-1} : X_{m-1} \inj X_m$ and $j_{m} : U_{m} = (X_m \setminus X_{m-1}) \inj 
X_m$ be the closed and open embeddings.
We consider the localization exact sequence for $MGL'$ ({\sl cf.} 
\cite[Page~35]{Levine1}).
\begin{equation}\label{eqn:split*1}
\cdots \to MGL'_{*,*}\left(X_{m-1}\right) \xrightarrow{i_{(m-1)*}}
MGL'_{*,*}\left(X_m\right) \xrightarrow{j^*_m}
MGL'_{*,*}\left(U_m\right) \to \cdots.
\end{equation}
Using ~\eqref{eqn:split0}, it suffices now to construct a canonical splitting 
of the smooth pull-back $j^*_m$ in order to prove the lemma.  

Let $V_m \subset U_m \times Z_m$ be the graph of the projection 
$U_m \xrightarrow{{\phi}_m} Z_m$ and let $\ov{V}_m$ 
denote the closure of $V_m$ in $X_m \times Z_m$.  
Let $Y_m \to \ov{V}_m$ be a resolution of singularities. Since $V_m$ is 
smooth, we see that $V_m \stackrel{\ov{j}_m}{\inj} Y_m$ as an open subset. We 
consider the composite maps
\[
p_m : V_m \inj U_m \times Z_m \to U_m, \ \ 
q_m : V_m \inj U_m \times Z_m \to Z_m \ \ {\rm and} 
\]
\[
{\ov{p}}_m : Y_m \to X_m \times Z_m \to X_m, \ \ 
{\ov{q}}_m : Y_m \to X_m \times Z_m \to Z_m. 
\]
Note that ${\ov{p}}_m$ is a projective  morphism since $Z_m$ is projective.
The map $q_m$ is smooth and $p_m$ is an isomorphism.

We consider the diagram
\begin{equation}\label{eqn:split1}
\xymatrix{
{MGL'_{*,*}\left(Z_m\right)} \ar[r]^{{\ov{q}}^*_m} 
\ar[d]_{{\phi}^*_m}^{\cong} &
{MGL'_{*,*}\left(Y_m\right)} \ar[d]^{{{\ov{p}}_m}_*} \\
{MGL'_{*,*}\left(U_m\right)} & 
{MGL'_{*,*}\left(X_m\right)} \ar[l]^{j^*_m}.}
\end{equation} 
Since each $Z_m$ is smooth by our assumption, the map ${{\phi}^*_m}$ is an 
isomorphism by the homotopy invariance of the $MGL$-theory.
The map ${\ov{q}}^*_m$ is the pull-back $MGL'_{*,*}\left(Z_m\right) \cong
MGL^{*,*}\left(Z_m\right) \to MGL^{*,*}\left(Y_m\right) \cong
MGL'_{*,*}\left(Y_m\right)$ between the Voevodsky's motivic cobordism
of smooth varieties. It suffices to show that
this diagram commutes. For, the map ${{\ov{p}}_m}_* \circ {\ov{q}}^*_m \circ
{{\phi}^*_m}^{-1}$ will then give the desired splitting of the map
$j^*_m$. 

We now consider the following commutative diagram.
\[
\xymatrix{
X_m & U_m \ar[l]_{j_m}& \\
Y_m \ar[u]^{{\ov{p}}_m} \ar[dr]_{{\ov{q}}_m} & V_m \ar[u]_{p_m} \ar[d]^{q_m}
\ar[l]^{{\ov{j}}_m} & U_m \ar[ul]_{id} 
\ar[l]^{({\phi}_m, id)} \ar[dl]^{{\phi}_m} \\
& Z_m. & }
\]
Since the top left square is Cartesian  and $j_m$ is an open immersion, we have
$j^*_m \circ {{\ov{p}}_m}_*  = {p_m}_* \circ {\ov{j}}^*_m$ by the functoriality
property of the pull-back of $MGL'$-theory with respect to an open immersion. 
Now, using the fact that $({\phi}_m, id)$ is an isomorphism, 
we get 
\[
\begin{array}{lllll}
j^*_m \circ {{\ov{p}}_m}_* \circ {\ov{q}}^*_m & = & 
{p_m}_* \circ {\ov{j}}^*_m \circ {\ov{q}}^*_m & = &  
{p_m}_* \circ q^*_m \\
& = & {p_m}_* \circ {({\phi}_m, id)}_* \circ {({\phi}_m, id)}^* \circ
q^*_m & = & {id}_* \circ {\phi}^*_m \\
& = &  {\phi}^*_m. & &  
\end{array} 
\]
This proves the commutativity of ~\eqref{eqn:split1} and hence the proposition.
\end{proof}

\begin{lem}\label{lem:Niv-PB}
Let $Y \xrightarrow{f} X$ be a morphism of smooth $k$-varieties of relative
codimension $d$. Then the map $f^* : \Omega_p(X) \to \Omega_{p-d}(Y)$ preserves
the niveau filtration $(${\sl cf.} ~\eqref{eqn:niveau1}$)$.
\end{lem}
\begin{proof}
We need to show that $f^*\left(F_q\Omega_p(X)\right) \subset 
F_{q-d}\Omega_{p-d}(Y)$. We can factor $f$ as $Y \xrightarrow{i} U 
\xrightarrow{j} \P^n_X \xrightarrow{\pi} X$, where $i$ and $j$ are respectively
closed and open immersions, and $\pi$ is the usual projection.
One has then $f^* = i^* \circ j^* \circ {\pi}^*$ ({\sl cf.} 
\cite[5.1.3]{LM}). Since $\pi$ is the smooth projection and $j$ is an
open immersion, it is clear from the definition of the cobordism 
({\sl cf.} \cite[Theorem~2.3]{Krishna}) that 
${j^* \circ {\pi}^*} \left(F_q\Omega_p(X)\right) \subset 
F_{q+n}\Omega_{p+n}(U)$. Since $d = {\rm codim}_U(Y) - n$, we need to prove
the lemma when $f$ is a closed immersion. 

So let $\alpha = [W \xrightarrow{g} X]$ be a cobordism cycle where $W$ is an
irreducible and smooth variety and $g$ is projective. 
By \cite[Lemma~7.1]{LM1}, we can assume that $g$ is transverse to $f$.
In particular, $W ' = W \times_X Y$ is smooth and there is a Cartesian
square 
\begin{equation}\label{eqn:Niv-PB1}
\xymatrix@C.7pc{
W' \ar[r]^{f'} \ar[d]_{g'} & W \ar[d]^{g} \\
Y \ar[r]_{f} & X}
\end{equation}
in $\sV^S_k$, where the horizontal arrows are closed immersions and the
vertical arrows are projective. We then have 
\[
f^*(\alpha) =  f^* \circ g_*({\rm Id}_W) =  g'_* \circ {f'}^*({\rm Id}_W) 
= g'_*({\rm Id}_{W'}),  
\]
where the second equality follows from the transversality condition 
({\sl cf.} \cite[5.1.3]{LM}). Since $f$ and $f'$ have same codimension, the
lemma follows.
\end{proof}

\subsection{Equivariant cobordism of filtrable varieties}
\label{subsection:Filtrable}
We recall from \cite[Section~3]{Brion2} that a $k$-variety $X$ with an action 
of a split torus $T$ is called {\sl filtrable} if the fixed point locus $X^T$ 
is smooth and projective, 
and there is a numbering $X^T = \stackrel{n}{\underset{m=0}{\coprod}}
Z_m$ of the connected components of the fixed point locus, a 
filtration of $X$ by $T$-invariant closed subschemes
\begin{equation}\label{eqn:filtration-BB}
{\emptyset} = X_{-1} \subset X_0 \subset \cdots \subset X_n = X
\end{equation}
and maps ${\phi}_m : U_m = (X_m \setminus X_{m-1}) \to Z_m$ for $0 \le m \le n$ 
which are all $T$-equivariant vector bundles.  
The following celebrated theorem of Bialynicki-Birula \cite{BB}
(generalized to the case of non-algebraically closed fields by Hesselink
\cite{Hessel}) will be crucial for understanding the
equivariant cobordism of smooth projective varieties.
\begin{thm}[Bialynicki-Birula, Hesselink]\label{thm:BBH}
Let $X$ be a smooth projective variety with an action of $T$. Then $X$ is
filtrable.
\end{thm}
This is roughly proven by choosing a generic one parameter subgroup 
$\lambda \subset T$ such that $X^{\lambda} = X^T$ and taking the various strata 
of the filtration to be the locally closed $T$-invariant subsets of the form
\[
X_{+}(Z, \lambda) =
\{x \in X| \ {\underset{t \to 0}{\rm lim}} \ \lambda(t) \cdot x \in Z\},
\]
where $Z$ is a connected component of $X^T$. Note that the above limit
exists since $X$ is projective. The fiber of the equivariant vector bundle
$X_{+}(Z, \lambda) \to Z$ is the positive eigenspace $T^{+}_{z}X$ for the
$T$-action on the tangent space of $X$ at any point $z \in Z$.

\begin{lem}\label{lem:filter-Equiv}
Let $X$ be a filtrable variety with an action of a split torus $T$ 
as above. Then for every $0 \le m \le n$,
there is a canonical split exact sequence
\begin{equation}\label{eqn:filter-Equiv*}
0 \to \Omega^T_*(X_{m-1}) \xrightarrow{i_{(m-1)*}} \Omega^T_*(X_m) 
\xrightarrow{j^*_m} \Omega^T_*(U_m) \to 0.
\end{equation}
\end{lem}
\begin{proof} Let $r$ be the rank of $T$ and for each $j \ge 1$, let
$(V_j, U_j)$ be the good pair for the $T$-action corresponding to $j$ as 
chosen in the proof of Lemma~\ref{lem:trivial-T}.
Notice in particular that ${U_j}/{T} \cong \left(\P^{j-1}_k\right)^r$.
For any scheme $Y$ with $T$-action, we set $Y^j = Y \stackrel{T}{\times} U_j$
for every $j$. Given the $T$-equivariant filtration as in 
~\eqref{eqn:filtration-BB}, it is easy to see that for each $j$, there is an 
associated system of filtrations 
\begin{equation}\label{eqn:filter-Equiv1}
{\emptyset} = (X_{-1})^j \subset (X_0)^j \subset \cdots \subset (X_n)^j = X^j
\end{equation}
and maps ${\phi}_m : (U_m)^j = (X_m)^j \setminus (X_{m-1})^j \to (Z_m)^j$ for 
$0 \le m \le n$ which are all vector bundles. Observe also that
as $T$ acts trivially on each $Z_m$, we have that $(Z_m)^j 
\cong Z_m \times \left({U_j}/T\right) \cong Z_m \times \left(\P^{j-1}_k\right)^r$.
Since $Z_m$ is smooth and projective, this in turn implies that
$(Z_m)^j$ is a smooth projective variety. We conclude that the filtration
~\eqref{eqn:filter-Equiv1} of $X^j$ satisfies all the conditions of 
Proposition~\ref{prop:filter-Gen}. In particular, we get split exact
sequences 
\[
0 \to MGL'_{*,*}\left((X_{m-1})^j\right) \to  MGL'_{*,*}\left((X_{m})^j\right)
\to  MGL'_{*,*}\left((U_{m})^j\right) \to 0.
\]
Applying Theorem~\ref{thm:LCOMP}, we get for each $0 \le m \le n$,
$i \in \Z$ and $j \ge 1$,
the canonical split exact sequence
\[
0 \to \Omega_i\left((X_{m-1})^j\right) \xrightarrow{i_{(m-1)*}} 
\Omega_i\left((X_{m})^j\right) \xrightarrow{j^*_m} 
\Omega_i\left((U_{m})^j\right) \to 0.
\]
By \cite[Theorem~3.4]{Krishna4}, this sequence remains exact at each 
level of the niveau filtration.
We now claim that the inverse $(j^*_m)^{-1}$ also preserves this filtration.
By the construction of the inverse of $j^*_m$ in diagram~\eqref{eqn:split1},
we only need to show that ${{\ov{p}}_m}_*$ and ${\ov{q}}^*_m$ preserve
the niveau filtration. This holds for ${{\ov{p}}_m}_*$ by 
\cite[Lemma~3.3]{Krishna4}
and so is the case for ${\ov{q}}^*_m$ by Lemma~\ref{lem:Niv-PB}.
This proves the claim. 
We conclude that there are canonical split exact sequences
\[
0 \to {\Omega_i\left(X_{m-1}\right)}_j \xrightarrow{i_{(m-1)*}} 
{\Omega_i\left(X_{m}\right)}_j \xrightarrow{j^*_m} 
{\Omega_i\left(U_{m}\right)}_j \to 0. 
\]
Taking the inverse limit over $j$, we get the desired split exact sequence 
~\eqref{eqn:filter-Equiv*}.
\end{proof}

The following is the main result about the equivariant cobordism for the
torus action on smooth filtrable varieties. 

\begin{thm}\label{thm:Main-Str}
Let $T$ be a split torus of rank $r$ acting on a filtrable variety
$X$ and let $i: X^T =  \stackrel{n}{\underset{m=0}\coprod} Z_m \inj X$ 
be the inclusion of the fixed point locus. Then there is a canonical
isomorphism
\[
\stackrel{n}{\underset{m = 0}\bigoplus} 
\Omega^T_*(Z_m) \xrightarrow{\cong} \Omega^T_*(X).
\]
of $S$-modules.
In particular, there is a canonical isomorphism
\begin{equation}\label{eqn:Main-Str1}
\Omega^T_*(X) \xrightarrow{\cong} \Omega_*(X)[[t_1, \cdots , t_r]]
\end{equation}
of $S$-modules. 
\end{thm}
\begin{proof} By inducting on $n$, it follows from 
Lemma~\ref{lem:filter-Equiv} that there are canonical isomorphisms
\[
\stackrel{n}{\underset{m = 0}\bigoplus} \ 
\Omega^T_*(Z_m) \xrightarrow{\cong} \Omega^T_*(X) \ {\rm and} \
\ \stackrel{n}{\underset{m = 0}\bigoplus} \ \Omega_*(Z_m) 
\xrightarrow{\cong} \Omega_*(X)
\]
of $S$-modules. 
The second assertion now follows from these two isomorphisms and
Lemma~\ref{lem:trivial-T} since $T$ acts trivially on $X^T$.
\end{proof}

Recall from Subsection~\ref{subsection:FCL} that if $Y \to X$ is a 
$T$-equivariant projective morphism of $k$-schemes with $T$-action such that
$Y$ is smooth, then $[Y \to X]$ defines unique elements in $\Omega^T_*X)$
and $\Omega_*(X)$, called the fundamental classes of $[Y \to X]$. 
 
\begin{cor}\label{cor:FIXEDPOINTS}
Let $T$ be a split torus of rank $r$ acting on a filtrable variety
$X$ such that $X^T$ is the finite set of smooth closed points 
$\{x_0, \cdots , x_n\}$. For $0 \le m \le n$, let $f_m : \wt{X}_m \to X_m$ be a
$T$-equivariant resolution of singularities and let $\wt{x}_m$ be the
fundamental class of the $T$-invariant cobordism cycle
$[\wt{X}_{m} \to X]$ in $\Omega^T_*(X)$.
Then $\Omega^T_*(X)$ is a free $S$-module with a basis $\{\wt{x}_0,
\cdots , \wt{x}_n\}$. In particular, $\Omega_*(X)$ is a free $\bL$-module
with a basis $\{\wt{x}_0, \cdots , \wt{x}_n\}$.
\end{cor}
\begin{proof} It follows by inductively applying Lemma~\ref{lem:filter-Equiv} 
that $\{\wt{x}_0, \cdots , \wt{x}_n\}$ spans $\Omega^T_*(X)$ as $S$-module.
We show its linear independence by the induction on the length of the 
filtration. There is nothing to prove for $n = 0$ and so we assume that 
$n \ge 1$. Let $i_{(n-1)} : X_{n-1} \to X_n$ be the inclusion map. 
Notice that for $0 \le m \le n-1$, $\wt{x}_m$ defines a unique 
class $\wt{x}'_m$ in $\Omega^T_*(X_{n-1})$ and 
$\wt{x}_m = {i_{(n-1)}}_*(\wt{x}'_m)$.

Since $U_n$ is an affine space, it follows from Lemma~\ref{lem:filter-Equiv} 
that $\stackrel{n}{\underset{m = 0}\sum} a_m \wt{x}_m = 0$ implies that
$a_n j^*_n(\wt{x}_n) = 0$. Since $j^*_n(\wt{x}_n) = 1$, we conclude that
$a_n = 0$. It follows again from Lemma~\ref{lem:filter-Equiv} that
$\stackrel{n-1}{\underset{m = 0}\sum} a_m \wt{x}'_m = 0$ in 
$\Omega^T_*(X_{n-1})$ and hence $a_m = 0$ for each $m$ by induction. 
This proves the first assertion of the
corollary. The second assertion follows from the first and 
Theorem~\ref{thm:FF*}.
\end{proof}  

\subsection{Case of toric varieties}
Let $X = X(\Delta)$ be a smooth projective toric variety associated to a fan
$\Delta$ in $N = \Hom(\G_m, T)$. Then the fixed point locus $X^T$ is the
disjoint union of the closed orbits $O_{\sigma}$, where $\sigma$ runs over
the set of $r$-dimensional cones in $\Delta$. Since $X$ is projective and any
orbit is affine, we see that $\sigma$ is a maximal cone if and only if
$O_{\sigma}$ is a closed point and hence a fixed point. In other words,
$X^T = {\underset{\sigma \in \Delta_{\rm max}}\prod} O_{\sigma}
= \{x_0, \cdots , x_n\}$. Note also that all these closed points 
are in fact $k$-rational. For any $\sigma \in \Delta$, let
$V(\sigma)$ denote the closure of the orbit $O_{\sigma}$ in $X$.

It follows from Theorem~\ref{thm:Main-Str} that $\Omega^*_T(X)$ is a free
$S$-module of rank $|\Delta_{\max}|= n+1$. Moreover, it follows from
Corollary~\ref{cor:FIXEDPOINTS} that a
basis of this free module is given by $\{\wt{x}_0, \cdots ,\wt{x}_n\}$, where
$\wt{x}_m$ is the fundamental class of the $T$-invariant cobordism cycle
$[X_{m} \to X]$.

If we now fix an ordering $\{\sigma_0, \cdots , \sigma_n\}$ of
$\Delta_{\rm max}$ and let $\tau_m \subset \sigma_m$ be the cone which is the
intersection of $\sigma_m$ with all those $\sigma_j$ such that $j \ge m$
and which intersect $\sigma_m$ in dimension $r-1$, then we can choose
the ordering of $\Delta_{\rm max}$ such that 
\[
\tau_i \subset \sigma_j \ {\rm only \ if} \ i \le j.
\]
In this case, $X_m$ is same as $V(\tau_{n-m})$ for $0 \le m \le n$, which is
itself a smooth toric variety. Let $[V(\tau_i)]$ denote the fundamental
class of $[V(\tau_i) \to X]$. We conclude the following.

\begin{cor}\label{cor:T-FREE}
Let $X = X(\Delta)$ be as above. \\
$(i) \ \Omega^*_T(X)$ is free $S$-module with basis $\{[V(\tau_0)], \cdots ,
[V(\tau_n)]\}$. \\
$(ii) \ \Omega^*(X)$ is free $\bL$-module with basis $\{[V(\tau_0)], \cdots ,
[V(\tau_n)]\}$. \\
$(iii)$ There is an $MGL^{*,*}(k)$-linear isomorphism 
$MGL^{*, *}(X) \cong \left(MGL^{*, *}(k)\right)^{n+1}$.
\end{cor}
\begin{proof}
We have already shown $(i)$ above and $(ii)$ follows from $(i)$ and
Theorem~\ref{thm:FF*}. The last part follows from 
Proposition~\ref{prop:filter-Gen}.
\end{proof}

\subsection{Case of flag varieties}
Let $G$ be a connected reductive group with a split maximal torus $T$ and let
$B$ be a Borel subgroup of $G$ containing $T$. Let $X = G/B$ be the associated
flag variety of the left cosets of $B$ in $G$. Then the left multiplication
of $G$ on itself induces a natural $G$-action on $X$. In particular,
$X$ is a smooth projective $T$-variety and hence filtrable. The Bruhat
decomposition of $G$ implies that $X^T = {\underset{w \in W}\coprod} wB$,
where $W$ is the Weyl group of $G$ with respect to $T$. Notice here that
the coset $wB$ makes sense even though $w$ is not an element of $G$.
We can choose an ordering $\{w_0, \cdots, w_n\}$ of the elements of $W$ such 
that $X_m$ is the union of the Schubert varieties $X_{w_i}$ with $i \le m$.
For each $w \in W$, there is canonical $B$-equivariant resolution of 
singularities $f_w : \wt{X}_w \to X_w$ which is defined by inducting on the 
length of $w$ with $X_0$ being the closed point associated to the identity 
element of $W$. The varieties $\wt{X}_w$ are called the {\sl Bott-Samelson} 
varieties associated to $X$. We refer to \cite[Subsection~2.2]{Brion3} for a 
very nice exposition of these facts.  We shall call the fundamental classes
$[\wt{X}_w \to X]$ in $\Omega^*_T(X)$ as the Bott-Samelson cobordism classes.

It is known that these Bott-Samelson classes are not invariants of 
the associated Schubert variety $X_w$ and they depend on a reduced
decomposition of $w \in W$ in terms of the simple roots
({\sl cf.} \cite{BE}). We fix a choice $\{\wt{X}_w | w \in W\}$
of the Bott-Samelson varieties.
As an immediate consequence of Corollary~\ref{cor:FIXEDPOINTS}, we obtain
the following. It was earlier shown in \cite[Proposition~3.1]{HK} that the 
ordinary cobordism group $\Omega^*(X)$ is a free $\bL$-module. We shall 
use the following to compute the cobordism ring of $G/B$ later in this paper.  

\begin{cor}\label{cor:FLAG-BASIS}
Let $X = G/B$ be as above. Then $\Omega^*_T(X)$ is a free $S$-module
with a basis given by the Bott-Samelson classes. The similar conclusion 
holds for $\Omega^*(X)$.
\end{cor}

As another consequence of Theorem~\ref{thm:Main-Str}, we obtain the following
generalization of Brion's result \cite[Theorem~2.1]{Brion2} for smooth 
schemes. 

\begin{thm}\label{thm:Inv-Gen}
Let $T$ be a split torus acting on a smooth $k$-variety $X$. Then the 
$S$-module $\Omega^*_T(X)$ is generated by the fundamental classes
of the $T$-invariant cobordism cycles in $\Omega^*(X)$.
\end{thm}
\begin{proof} Since we are in characteristic zero, we can use the
canonical resolution of singularities to get a $T$-equivariant open
embedding $j : X \inj Y$, where $Y$ is a smooth and projective $T$-variety.
Since the restriction map $\Omega^*_T(Y) \xrightarrow{j^*} \Omega^*_T(X)$
is $S$-linear and surjective by Theorem~\ref{thm:Basic} $(ii)$, it suffices 
to prove the theorem when $X$ is projective. In this case, the assertion
is proved exactly like the proofs of Theorems~\ref{thm:BBH} and 
~\ref{thm:Main-Str} using an induction on the length of the filtration. 
\end{proof}
 
\begin{remk} In \cite[Theorem~2.1]{Brion2}, Brion also describes the
relations which explicitly describe the $T$-equivariant Chow groups
in terms of the invariant cycles. It will be interesting to know what are
the analogous relations for cobordism. We shall come back to this problem
in a subsequent work.
\end{remk}

\begin{exm}\label{exm:proj-space}
As an illustration of the above structure theorems, we deduce the formula
for the equivariant cobordism ring of the projective line where $\G_m$ acts
with weight $\chi$. We can write $\P^1_k = \A^1_{0} \cup \A^1_{\infty}$ as union of
$\G_m$-invariant affine lines where the first (resp. second) affine line is the
complement of $\infty$ (resp. $0$). We get pull-back maps
$\Omega^*_{\G_m}(\P^1) \xrightarrow{i^*_0} \Omega^*_{\G_m}(\A^1_0)
\xrightarrow{\cong} \Omega^*_{\G_m}(\{0\}) \cong S(\G_m)$. We similarly have
the pull-back map $i^*_{\infty}$. It follows from Lemma~\ref{lem:filter-Equiv}
that there is there is a short exact sequence of ring homomorphisms
\[
\xymatrix{
0 \ar[r] & \Omega^*_{\G_m}(\P^1) \ar[r]^<<<<<{(i^*_0, i^*_{\infty})} &
\Omega^*_{\G_m}(\A^1_0) \times \Omega^*_{\G_m}(\A^1_{\infty}) 
\ar[r]^{\ \ \ \ \ \ \ j^*_0 - j^*_{\infty}} & \Omega^*_{\G_m}(\G_m) \ar[r] & 0.} 
\]
Identifying the last term with $\Omega^*(k) \cong \bL$ and $S(\G_m)$ with
$\bL[[t]]$, we have a
short exact sequence of ring homomorphisms
\begin{equation}\label{eqn:SReis0}
\xymatrix{
0 \ar[r] & \Omega^*_{\G_m}(\P^1) \ar[r]^<<<<<{(i^*_0, i^*_{\infty})} &
\bL[[t]] \times \bL[[t]] 
\ar[r]^{\ \ \ \ \ \ \ j^*_0 - j^*_{\infty}} & \bL \ar[r] & 0.} 
\end{equation}
It follows from this exact sequence and Lemma~\ref{lem:elem1} that there is
a ring isomorphism
\begin{equation}\label{eqn:SReis}
\frac{\bL[[x,y]]}{(xy)} \xrightarrow{\cong} \Omega^*_{\G_m}(\P^1)
\end{equation}
where the images of the variables are the fundamental classes of $0$ and
$\infty$. This gives the Stanley-Reisner presentation for the equivariant
cobordism of the projective line.
\end{exm}

\section{Localization for torus action: First steps}
\label{section:Equi*-T}
The localization theorems are the most powerful tools in the study of the 
equivariant cohomology of smooth varieties with torus action.
They provide simple formulae to describe the relation between the equivariant 
cobordism ring of a smooth projective variety and that of the fixed point 
locus for the action of a torus. The localization theorems for the equivariant 
Chow groups and their important applications are considered in \cite{Brion2}. 
Our goal in the rest of this paper is to prove such results for the 
equivariant cobordism.
These results will be used in \cite{KK} to compute the equivariant and the
non-equivariant cobordism rings of certain spherical varieties. 
More applications of these results appear in \cite{KU}.
Below, we prove some algebraic results that we need for the localization
theorems. 

Let $R$ be a commutative Noetherian ring and let 
$A = {\underset{j \in \Z}\oplus} A_j$ be a commutative graded
$R$-algebra with $R \subset A_0$. 
Let $S^{(n)} = {\underset{i \in \Z} \oplus} S_i$ be the graded power series
ring $A[[{\bf t}]] : = A[[t_1, \cdots , t_n]] $ 
({\sl cf.} Section~\ref{section:AC}).
Recall that $S_i$ is the set of formal power series of the form $f({\bf t}) = 
{\underset{m({\bf t}) \in \sC} \sum} a_{m({\bf t})}  m({\bf t})$
such that $a_{m({\bf t})}$ is a homogeneous element in $A$
and $|a_{m({\bf t})}| + |m({\bf t})| = i$. Here, $\sC$ is the set of all 
monomials in ${\bf t} = (t_1, \cdots , t_n)$ and 
$|m({\bf t})| = i_1 + \cdots + i_n$ if 
$m({\bf t}) = t^{i_1}_1 \cdots t^{i_n}_n$. Notice that 
$S^{(n-1)}[[t_i]] \cong S^{(n)}$ and ${S^{(n)}}/{(t_i)} \cong
S^{(n-1)}$. 

Recall from \cite{LM} that a formal (commutative) group law over
$A$ is a power series $F(u,v) \in A[[u,v]]$ such that \\
$(i)$ \ $F(u, 0) = F(0,u) = u \in A[[u]]$ \\
$(ii)$ \ $F(u,v) = F(v,u)$ \\
$(iii)$ \ $F\left(u, F(v,w)\right) = F\left(F(u,v), w\right)$ \\
$(iv)$ there exists (unique) $\rho(u) \in A[[u]]$ such that 
$F\left(u, \rho(u)\right) = 0$. \\
We write $F(u,v)$ as $u+_F v$. We shall denote $\rho(u)$ by $[-1]_Fu$.
Inductively, we have $[n]_Fu = [n-1]_Fu +_F u$ if $n \ge 1$ and $[n]_Fu =
[-n]_F\rho(u)$ if $n \le 0$. The sum $\stackrel{m}{\underset{i =1}\sum}
[n_i]_Fu_i$ will mean $[n_1]_Fu_1 +_F \cdots +_F [n_m]_Fu_m$ for $n_i \in \Z$.
It is known that such a power series is of the form
\begin{equation}\label{eqn:FGL**}
F(u, v) = (u + v) + uv \left({\underset{i,j \ge 1}\sum} 
a_{i,j} u^{i-1}v^{j-1}\right)
\in A[[u,v]], \ {\rm where} \ a_{i,j} \in A_{1-i-j}.
\end{equation} 
Notice that $F(u,v)$ is a homogeneous element of degree one in
the graded power series ring $A[[u,v]]$ if $u, v$ are homogeneous elements
of degree one. It follows from the above conditions that
$\rho(u) = -u + u^2{\underset{j \ge 0}\sum} b_ju^j$. We conclude that
for a set $\{u_1, \cdots, u_m\}$ of homogeneous elements of degree one,
one has 
\begin{equation}\label{eqn:FG*1}
\begin{array}{lll}
\stackrel{m}{\underset{i =1}\sum} [n_i]_Fu_i & = &
\stackrel{m}{\underset{i =1}\sum} n_iu_i + {\underset{|m({\bf u})| \ge 2}\sum}
a_{m({\bf u})} m({\bf u})
\end{array}
\end{equation}
is also homogeneous and of degree one. For the rest of this text, 
our ring $A$ will be a graded $\bL$-algebra and $F(u,v)$ will be the formal 
group law on $A$ induced from that of the universal formal group law 
$F_{\bL}$ on $\bL$. We recall the following inverse function theorem
for a power series ring.

\begin{lem}\label{lem:IFT}
Let $A$ be any commutative ring and let $\{f_1, \cdots , f_n\}$ be a set of
power series in the formal power series ring 
$\widehat{A[[{\bf t}]]}$ such that 
$((\frac{\partial f_i}{\partial t_j}))(0) \in GL_n(A)$.
Then the $A$-algebra  homomorphism
\[
\phi : \widehat{A[[{\bf t}]]} \to \widehat{A[[{\bf t}]]}, 
\ \ \phi(t_j) = f_j
\]
is an isomorphism.
\end{lem}
\begin{proof}
{\sl Cf.} \cite[Exercise~7.25]{Eisenbud}.
\end{proof}

\begin{lem}\label{lem:NZD}
Let $A$ and $S^{(n)}$ be as above such that no non-zero element of 
$R$ is a zero divisor in $A$. Let 
$f = \stackrel{n}{\underset{i=1}\sum}[m_i]_Ft_i$
be a non-zero homogeneous element of degree one in $S^{(n)}$. Let 
$g(f) = f^r + {\alpha}_{r-1}f^{r-1} + \cdots + {\alpha}_1f + {\alpha}_0 
\in S^{(n)}$ be such that each ${\alpha}_j$
is homogeneous of degree $r-j$ in $A$ and is nilpotent. 
Then $g(f)$ is a non-zero divisor in $S^{(n)}$.
\end{lem}
\begin{proof} Since $S^{(n)}$ is a subring of the formal power series ring 
$\widehat{A[[{\bf t}]]}$, it suffices to show that $g(f)$ is a non-zero 
divisor in $\widehat{A[[{\bf t}]]}$. We can thus assume that $S^{(n)}$ is the 
formal power series ring. 

We prove the lemma by induction on $n$. We first assume that $n=1$, in which
case $S^{(1)} = \widehat{A[[t]]}$ and $f = [m]_Ft = mt + t^2f'(t)$ with 
$m \neq 0$ by ~\eqref{eqn:FG*1}. 

It follows from our assumption that every non-zero element of $R$ is a 
non-zero divisor in $S^{(1)}$. Since the map $S^{(1)} \to S^{(1)}[m^{-1}]$ is then 
an injective map of rings, it suffices to show that
$g(f)$ is not a zero divisor in $S^{(1)}[m^{-1}]$. That is, we can assume that 
$m = 1$. It follows then from Lemma~\ref{lem:IFT} that there is an
$A$-automorphism of $S^{(1)}$ which takes $f$ to $t$. Thus we can assume
that $f= t$ and $g(f) = g(t) = t^r + {\alpha}_{r-1}t^{r-1} + \cdots + 
{\alpha}_1t + {\alpha}_0$.   

Now, let $g'(t) = \stackrel{\infty}{\underset{j = 0}\sum}
a_jt^j$ be a power series in $S$. Then,
\[
\begin{array}{lll}
g(t) g'(t) = 0 & \Rightarrow &\stackrel{\infty}{\underset{j = 0}\sum}
\left(a_{j-r} + a_{j-r+1} \alpha_{r-1} + \cdots + a_{j-1}\alpha_1 + a_j\alpha_0
\right) t^j = 0 \\
& \Rightarrow & a_j = - \stackrel{r-1}{\underset{i = 0}\sum} a_{j+r-i} \alpha_{i}
\ \forall \ j \ge 0. \hspace*{3cm} (*) 
\end{array}
\]
Applying $(*)$ recursively, we see that each $a_j$ can be expressed as 
an $A$-linear combination of monomials in $\alpha_i$'s of arbitrarily large
degree in $A$. Since each $\alpha_i$ is of positive degree and nilpotent, 
we must have $a_j = 0$ for $j \ge 0$. 

To prove the general case, suppose the lemma is proven when the number
of variables is strictly less than $n$ with $n \ge 2$.
Suppose $g'({\bf t}) \in S$ is such that $g(f) g'({\bf t}) = 0$. 
If $g'({\bf t})$ is not zero, we remove those terms $a_{m({\bf t})}  m({\bf t})$ 
from its expression for which $a_{m({\bf t})} = 0$. In particular,
we get that none of the coefficients of $g'({\bf t})$ is zero. We show that this
leads to a contradiction. 

Set $m_0({\bf t}) = t_1t_2 \cdots t_n$ and write $g'({\bf t}) = 
\left(m_0({\bf t})\right)^p h({\bf t})$ such that $h({\bf t})$ is not
divisible by $m_0({\bf t})$. Then, there is a term $a_{m'({\bf t})} m'({\bf t})$ of
$h({\bf t})$ which is not divisible by at least one $t_i$. By permutation,
we can assume it is not divisible by $t_n$. 
If $m_i = 0$ for $1 \le i \le n-1$, then  $f$ is of the form
$f = [m_n]_Ft_n = m_nt_n + t^2_n{\underset{j \ge 0}\sum}b_jt^j_n$ with 
$m_n \neq 0$ by ~\eqref{eqn:FG*1}. Hence,
we are in the situation of the above one variable case with $A$ replaced by
$A[[t_1, \cdots , t_{n-1}]]$. So we assume that $m_i \neq 0$ for some $i \neq n$.
  
Since $m_0({\bf t})$ is a non-zero divisor in $S$, we must have 
$g(f)h({\bf t}) = 0$. Let $\ov{h({\bf t})}$ be the image of a power series
$h({\bf t})$ under the quotient map $S^{(n)} \surj S^{(n-1)}$. Then we get
${\ov{g(f)}} \ {\ov{h({\bf t})}} = 0$ in $S^{(n-1)}$. By our choice, the term
$a_{m'({\bf t})} m'({\bf t})$ still survives in $S^{(n-1)}$. On the other hand,
as $\ov{g(t)}$ is of the same form as $g(t)$, the induction implies that
$\ov{h({\bf t})}$ must be zero. In particular, we must have 
$a_{m'({\bf t})} = 0$. Since the coefficients of $h({\bf t})$
are same as those of $g'({\bf t})$, we arrive at a contradiction.
\end{proof}

The following is a variant of Lemma~\ref{lem:NZD} and is proved in the
similar way.
\begin{lem}\label{lem:NZDS}
Let $A$ and $S^{(n)}$ be as above such that no non-zero element of 
$R$ is a zero divisor in $A$. Let 
$f = \stackrel{n}{\underset{i=1}\sum}[m_i]_Ft_i$
be a non-zero homogeneous element of degree one in $S^{(n)}$. Let 
$v \in A$ be a homogeneous element of degree one which is nilpotent.
Then $F(f,v)$ is a non-zero divisor in $S^{(n)}$.
\end{lem}
\begin{proof} The proof is exactly along the same lines as that of
Lemma~\ref{lem:NZD}. We give the main steps. We can assume as before that 
$S^{(n)}$ is the formal power series ring. It follows from ~\eqref{eqn:FGL**}
that $F(f, v)$ is of the form $f + v\left(1 + fF'(f,v)\right)$, where
$1 + fF'(f,v)$ is an invertible element. In particular,
we can write $F(f, v) = uf + v$, where 
$u = \left(1 - fF' + f^2 {F'}^2 - \cdots\right)$
is an invertible element in $S^{(n)}$. As before, we prove by
induction on $n$. In case of $n =1$, we can write $S^{(1)} = A[[t]]$  and
$f = mt + t^2f'(t)$. We can further assume that $m =1$.  
In particular, we get $uf = t + t^2f''(t)$. It follows from Lemma~\ref{lem:IFT}
that there an $A$-automorphism of $S^{(1)}$ which takes $uf$ to $t$.
This reduces to showing that $t+v$ is a non-zero divisor in $S^{(1)}$,
which is shown in Lemma~\ref{lem:NZD}.

Now suppose the result holds for all $m \le n-1$ with $n \ge 2$.
If $m _i = 0$ for $1 \le i \le n-1$, then $f = [m_n]_Ft_n$ with $m_n \neq 0$
and hence we are in the situation of one variable case. So we can
assume that $m_i \neq 0$ for some $i \neq n$. Since the image $\ov{f}$ of $f$ 
under the natural surjection $S^{(n)} \surj S^{(n-1)}$ is of the same form as 
that of $f$ and $\ov{f} \neq 0$, we argue as in the proof of Lemma~\ref{lem:NZD}
to conclude that $F(f, v)$ is a non-zero divisor in $S^{(n)}$.
\end{proof}

\begin{lem}\label{lem:elem1}
Let $T$ be split torus of rank $n$ and let $\{\chi_1, \cdots , \chi_s\}$ be 
set of characters of $T$ such that for $j \neq j'$, the set
$\{\chi_j, \chi_{j'}\}$ is a part of a basis of $\widehat{T}$.
Let ${\gamma}_j = \left(c^T_1(L_{\chi_j})\right)^{d_j}$ with $d_j \ge 0$.
Then 
\[
\left({\gamma}_1 \cdots {\gamma}_s\right) \ = \
\stackrel{s}{\underset{j=1}{\bigcap}} \left({\gamma}_j\right)
\]
as ideals in $S(T)$.
\end{lem}  
\begin{proof} By ignoring those $j \ge 1$ such that $d_j = 0$, we can
assume that $d_j \ge 1$ for all $j$.
Using a simple induction, it suffices to show that for 
$j \neq j'$, the relation ${\gamma}_j | q{\gamma}_{j'}$  
implies that ${\gamma}_j | q$. So we can assume that $s =2$.
By assumption, we can extend $\{\chi_1, \chi_2\}$ to a basis 
$\{\chi_1, \cdots , \chi_n\}$ of $\widehat{T}$. In particular, we
can identify $S(T)$ with the graded power series ring $\bL[[t_1, \cdots , t_n]]$
such that $t_j = c^T_1(L_{\chi_j})$ ({\sl cf.} ~\eqref{eqn:CBT*}).
We can thus write $\gamma_j = t^{d_j}_j$ for $j =1, 2$.

We now write $S(T)$ as $A[[t_1]]$ where $A = \bL[[t_2, \cdots , t_n]]$
({\sl cf.} Lemma~\ref{lem:GPSR}) and suppose that $p\gamma_2 = q\gamma_1$
in $S(T)$. We can write $p = s\gamma_1 + r$, where $r$ is a polynomial in
$t_1$ with coefficients in $A$ of degree less than $d_1$.

We now claim that $q- s\gamma_2 = 0$, which will finish the proof of the
lemma. Set $g(t_1) = q- s\gamma_2 = {\underset{j \ge 0}\sum} a_jt^j_1$.
This yields $g(t_1)\gamma_1 = \stackrel{\infty}{\underset{j = 0}\sum}
a_j t^{d_1 + j}_1$. Since $g(t_1)\gamma_1 = r \gamma_2$ is of degree less than 
$d_1$ in $t_1$, we see that $a_j = 0$ for all $j \ge 0$. 
\end{proof}

\section{Chern classes of equivariant bundles}\label{section:Chern-C}
We now apply the above algebraic results to deduce some consequences
for the Chern classes of the equivariant vector bundles on smooth varieties
with a torus action. We need the following
description of the equivariant cobordism for the trivial action of a torus
which follows essentially from the definitions.
Recall that for a split torus $T$ of rank $n$, $S =
\bL[[t_1, \cdots , t_n]]$ denotes the cobordism ring of the classifying space
of $T$.

\begin{lem}\label{lem:trivial-T}
Let $T$ be a split torus of rank $n$ acting trivially on a smooth
variety $X$ of dimension $d$ and let $\{\chi_1, \cdots , \chi_n\}$ be a chosen
basis of $\widehat{T}$. Then the assignment $t_i \mapsto c^T_1(L_{\chi_i})$
induces an  isomorphism of graded rings
\begin{equation}\label{eqn:trivial-T1}
\Omega^*(X)[[t_1, \cdots , t_n]] \xrightarrow{\cong} \Omega_T^*(X).
\end{equation}
\end{lem}   
\begin{proof} 
This is a direct application of the projective bundle formula in the
non-equivariant cobordism and is similar to the calculation of
$\Omega^*_T(k)$ ({\sl cf.} \cite[Example~6.4]{Krishna4}). We give the sketch. 

The chosen basis $\{\chi_1, \cdots , \chi_n\}$ of $\widehat{T}$
equivalently yields a decomposition $T = T_1 \times  \cdots  \times T_n$ 
with each $T_i$ isomorphic to $\G_m$ such that $\chi_i$ is a generator of 
$\widehat{T_i}$.
Let $L_{\chi}$ be the one-dimensional representation of $T$, where $T$ acts via 
$\chi$. For any $j \ge 1$, we take the good pair $(V_j, U_j)$ such that
$V_j = \stackrel{n}{\underset{i = 1} \prod} 
L^{\oplus j}_{\chi_i}$, $U_j = \stackrel{n}{\underset{i = 1} \prod} 
\left(L^{\oplus j}_{\chi_i} \setminus \{0\}\right)$ and 
$T$ acts on $V_j$ by $(t_1, \cdots , t_n)(x_1, \cdots , x_n)
= \left(\chi_1(t_1)(x_1), \cdots , \chi_n(t_n)(x_n)\right)$. 
Since $T$ acts trivially on $X$, it is easy to see that
$X \stackrel{T}{\times} U_j \cong X \times \left(X_1 \times \cdots \times 
X_n\right)$ with each $X_i$ isomorphic to $\P^{j-1}_k$.
Moreover, the $T$-line bundle $L_{\chi_i}$ gives the line bundle 
$L_{\chi_i} \stackrel{T_i}{\times} \left(L^{\oplus j}_{\chi_i} \setminus \{0\}\right)
\to X_i$ which is $\sO(\pm 1)$. Letting $\zeta_i$ be the first
Chern class of this line bundle, it follows from the projective bundle formula 
for the non-equivariant cobordism and Theorem~\ref{thm:NO-Niveu} that
\[
\Omega^i_{T}(X) = {\underset{p_1, \cdots , p_n  \ge 0} \prod} 
\Omega^{i-(\stackrel{n}{\underset{i=1} \sum} p_i)}(X) \zeta^{p_1}_1 \cdots \zeta^{p_n}_n,
\]
which is isomorphic to the set of formal power series in 
$\{\zeta_1, \cdots , \zeta_n\}$ of degree $i$ with coefficients in 
$\Omega^*(X)$. The desired isomorphism of ~\eqref{eqn:trivial-T1} is easily 
deduced from this.
\end{proof}

Let $T$ be a split torus of rank $n$ and let $\{\chi_1, \cdots , \chi_n\}$
be a chosen basis of the character group $M = \widehat{T}$. We have then seen 
that there is a
graded ring isomorphism $S = \Omega^*_T(k) \cong \bL[[t_1, \cdots , t_n]]$,
where $t_j = c^T_1(L_{\chi_j})$ for $1 \le j \le n$. It follows from
~\eqref{eqn:FGL} and the formula $c^T_1(L_1 \otimes L_2) = 
c^T_1(L_1) +_F c^T_1(L_2)$ 
that $c^T_1(L_{\chi^m_j}) = mt_j + t^2_j {\underset{i \ge 0}\sum} a_i t^i_j$.
In particular, we have for a character
$\chi = \stackrel{n}{\underset{j = 1}\prod} \chi^{m^j}_j$, 
\begin{equation}\label{eqn:F-Chern}
c^T_1(L_{\chi}) \ = \ \stackrel{n}{\underset{j =1}\sum} [m_j]_Ft_j \ = \ \ \ \
\stackrel{n}{\underset{j =1}\sum} m_jt_j +
{\underset{|m({\bf t})| \ge 2}\sum} a_{m({\bf t})} m({\bf t}).
\end{equation}
Let $S(T)[M^{-1}]$ denote the ring obtained by inverting all non-zero 
linear forms $\stackrel{n}{\underset{j =1}\sum} m_jt_j$ in $S(T)$.
Then $S(T)[M^{-1}]$ is also a graded ring and letting $f = 
\stackrel{n}{\underset{j =1}\sum} m_jt_j$ in ~\eqref{eqn:F-Chern},
we can write
\begin{equation}\label{eqn:F-Chern1}
c^T_1(L_{\chi}) \ = f \left(1 + f^{-1}
{\underset{|m({\bf t})| \ge 2}\sum} a_{m({\bf t})} m({\bf t})\right).
\end{equation}
in $S(T)[M^{-1}]$. Notice then that the term inside the parenthesis is
a homogeneous element of degree zero in $S(T)[M^{-1}]$ which is invertible.
We conclude in particular that $c^T_1(L_{\chi})$ is an invertible element
of $S(T)[M^{-1}]$. For a smooth $k$-scheme $X$ with an action of a split torus 
$T$, we shall write $\Omega^*_T(X) \otimes_S S[M^{-1}]$ as $\Omega_T^*(X)[M^{-1}]$.

\begin{lem}[Splitting Principle]\label{lem:SPLIT}
Let $T$ be a split torus of rank $n$ acting trivially on a smooth scheme $X$
and let $\{E_1, \cdots, E_s\}$ be a finite collection of vector bundles on $X$.
Then there is a smooth morphism $p: Y \to X$ which is a composition of
affine and projective bundle morphisms such that \\
$(i)$ \ $p^*(E_i)$ is a direct sum of line bundles for $1 \le i \le s$, \\
$(ii)$  the pull-back map $p^*: \Omega^*_T(X) \to \Omega^*_T(Y)$ is split
injective, where $T$ acts trivially on $Y$, and \\  
$(iii)$  the map $\frac{\Omega^*_T(X)}{\left(c^T_i(E)\right)} \to 
\frac{\Omega^*_T(Y)}{\left(p^*\left(c^T_i(E)\right)\right)}$ is split injective 
for any $T$-equivariant vector bundle $E$ on $X$ and any $i \ge 0$.
\end{lem} 
\begin{proof}
If $E = E_1$ is a single vector bundle on $X$, then it is shown in
\cite[Lemma~3.24]{Panin} that there are maps $X_2 \xrightarrow{p_2} X_1
\xrightarrow{p_1} X$ such that $p_1$ is the natural projection
$X_1 = \P(E) \to X$ and $p_2$ is an affine bundle map. 
Furthermore, $(p_1 \circ p_2)^*(E) = E' \oplus E''$, where $E'$ is a line
bundle. By repeating this process finitely many times, we see that there
is a smooth map $p : Y \to X$ which is composition of projective and affine
bundle morphisms and such that the condition $(i)$ of the lemma holds
on $Y$. 

Letting $T$ act trivially on all the intermediate schemes mapping to $X$, we 
see that all the intermediate maps are $T$-equivariant. Since an affine bundle
keeps the (equivariant) cobordism rings invariant, we only need to show 
$(ii)$ and $(iii)$ when $p : Y \to X$ is a projective bundle morphism.

So let $Y = \P(V) \xrightarrow{p} X$ be a such a projective bundle of relative 
dimension $r$ and let $\zeta \in \Omega^1(Y)$ be the first Chern class of the
tautological bundle $\sO_Y(-1)$ on $Y$. Then $\zeta = c^T_1\left(\sO_Y(-1) 
\otimes L_{\chi_0}\right)$, where $\chi_0$ is the trivial character of $T$.
It also follows from Lemma~\ref{lem:trivial-T} that $p^*$ is the natural
map of the graded power series rings $p^* : \Omega^*(X)[[t_1, \cdots , t_n]]
\to \Omega^*(Y)[[t_1, \cdots , t_n]]$. Moreover, the projective bundle
formula for the ordinary cobordism and the projection formula for the
equivariant cobordism ({\sl cf.} Theorem~\ref{thm:Basic} $(vii)$) imply that 
\begin{equation}\label{eqn:PNE}
p_*\left(\zeta^r \cdot p^*(x)\right) =  x \cdot p_*(\zeta^r) = x
\end{equation}
for any $x \in \Omega^*_T(X)$. This immediately implies the condition
$(ii)$ of the lemma. 

To prove the final assertion, we first deduce from the standard
properties of the Chern classes that $p^*\left(c^T_i(E)\right) =
c^T_i\left(p^*(E)\right)$ for any $T$-equivariant vector bundle $E$ on $X$. 
Setting $x = c^T_i(E) \in \Omega^*_T(X)$ and $y = c^T_i\left(p^*(E)\right)$,
we conclude that $p^*$ induces the map 
\[
\ov{p^*} : \frac{\Omega^*_T(X)}{\left(x\right)} \to 
\frac{\Omega^*_T(Y)}{\left(y\right)}.
\]
Let $q_* : \Omega^*_T(Y) \to \Omega^*_T(X)$ be the map 
$q(z) = p_*\left(\zeta^r \cdot z\right)$. Then $q_*$ is clearly 
$\Omega^*_T(X)$-linear. Moreover, it follows from the projection formula
and ~\eqref{eqn:PNE} that $q_*$ induces maps
\[
\frac{\Omega^*_T(X)}{\left(x\right)} \xrightarrow{\ov{p^*}}
\frac{\Omega^*_T(Y)}{\left(y\right)} \xrightarrow{\ov{q_*}}
\frac{\Omega^*_T(X)}{\left(x\right)} 
\]
such that the composite is identity. This completes the proof of the lemma.
\end{proof}

\noindent
{\bf Notation:} All results in the rest of this paper will be proven with
the rational coefficients. In order to simplify our notations, an abelian
group $A$ from now on will actually mean the $\Q$-vector space 
$A \otimes_{\Z} \Q$, and an inverse limit of abelian groups will mean the
limit of the associated $\Q$-vector spaces. In particular, all cohomology groups
will be considered with the rational coefficients and 
$\Omega^G_i(X) = {\underset{j}\varprojlim} \ {\Omega^G_i(X)}_j$,
where each ${\Omega^G_i(X)}_j$ at $j$th level is the non-equivariant cobordism 
group with rational coefficients.

\begin{lem}\label{lem:RIGID*}
Let $T$ be a split torus of rank $n$ acting trivially on a smooth variety $X$ 
and let $E$ be a $T$-equivariant vector bundle of rank $d$ on $X$ 
which is a direct sum of line bundles. Assume that
in the eigenspace decomposition of $E$ with respect to 
$T$, the submodule corresponding to the trivial character is zero.
Then $c^T_d(E)$ is a non-zero divisor in $\Omega^*_T(X)$ and is invertible
in $\Omega^*_T[M^{-1}]$.
\end{lem}
\begin{proof}
To prove the first assertion, it suffices to show by Lemma~\ref{lem:trivial-T} 
that $c^T_d(E)$ is a 
non-zero divisor in $\Omega^*(X)[[t_1, \cdots , t_n]]$. As in the proof
of Lemma~\ref{lem:NZD}, it suffices to show that 
$c^T_d(E)$ is a non-zero divisor in the formal power series ring
$\widehat{\Omega^*(X)[[{\bf t}]]}$, where 
${\bf t} = \left(t_1, \cdots , t_n\right)$.

Let $q : X \to {\rm Spec}(k)$ be the structure map.
Since $T$ acts on $X$ trivially, $E$ has a unique direct
sum decomposition 
\[
E = \ \stackrel{m}{\underset {i=1}{\oplus}} E_i \otimes q^*(L_{{\chi}_i}),
\]
where each $E_i$ is an ordinary vector bundle on $X$ and $L_{{\chi}_i}$ is the 
line bundle in ${\rm Pic}_T(k)$ corresponding to a 
character ${\chi}_i$ of $T$. 
Since ${\rm rank}(E) = d$, the Whitney sum formula ({\sl cf.}
Theorem~\ref{thm:Basic}, \cite[Proposition~4.1.15]{LM}) yields   
\[
c^T_d(E) = \stackrel{m}{\underset {i=1}{\prod}}c^T_{d_i}
\left(E_i \otimes q^*\left(L_{{\chi}_i}\right)\right).
\]

We can thus assume that 
$E = E_{\chi} \otimes q^*(L_{\chi})$, where $\chi$ is
not a trivial character by our assumption. 
Since $\Pic(BT) \cong \Q^n$ and since $\chi$ is
not trivial, we can extend $\chi$ to a basis of $\Q^n$.
We can now use ~\eqref{eqn:F-Chern} and Lemma~\ref{lem:IFT} to assume that
$c^T_1(L_{\chi}) = t_1$. 

We can also write $E = \stackrel{s}{\underset{i =1}\oplus} L_i$, where each
$L_i$ is a line bundle on $X$. The Whitney sum formula again implies that
\begin{equation}\label{eqn:NZD1*}
\begin{array}{lll}
c^T_d\left(E \otimes L_{\chi}\right) & = & 
c^T_d\left(\stackrel{s}{\underset{i=1}\oplus} 
\left(L_i \otimes L_{\chi}\right)\right) \\
& = & \stackrel{s}{\underset{i =1}\prod} c^T_1\left(L_i \otimes L_{\chi}\right) \\
& = & \stackrel{s}{\underset{i =1}\prod} F\left(t_1, c_1(L_i)\right),
\end{array}
\end{equation}
where $F(u,v)$ is the universal formal group law on $\bL$. Since
$c_1(L) \in \Omega^1(X)$ and since $\Omega^{> {\rm dim}(X)}(X) = 0$,
we see that $c_1(L_i)$ is nilpotent in $\Omega^*(X)$. In 
fact, this implies that $c_j(E)$ is nilpotent in $\Omega^*(X)$ for all vector 
bundles $E$ on $X$ and for all $j \ge 1$. We now apply Lemma~\ref{lem:NZDS}
with $R = \Q$ and $A = \Omega^*(X)$ to conclude that 
$c^T_d\left(E \otimes L_{\chi}\right)$ is a non-zero divisor in 
$\widehat{\Omega^*(X)[[{\bf t}]]}$.

To prove the invertibility of $c^T_d(E)$ in $\Omega^*_T[M^{-1}]$, we can again
use the above reductions to assume that $E = L \otimes L_{\chi}$ where $\chi$
is not a trivial character of $T$. In this case,
$F(f, v)$ is of the form $f + v\left(1 + fF'(f,v)\right)$, where
$1 + fF'(f,v)$ is an invertible element of degree zero. In particular,
we can write $F(f, v) = uf + v$, where 
$u = \left(1 - fF' + f^2 {F'}^2 - \cdots\right)$
is an invertible element of degree zero in $S^{(n)}$. 
In particular, $v$ is nilpotent. We have seen in ~\eqref{eqn:F-Chern1} that 
$f$ is invertible in $S[M^{-1}]$. Since $u$ is invertible in $\Omega^*_T(X)$
and $v$ is nilpotent, it follows that $uf+v$ is invertible in 
$\Omega^*_T[M^{-1}]$. 
\end{proof}

\begin{lem}\label{lem:CP}
Let $T$ be a split torus of rank $n$ acting trivially on a smooth variety
$X$. Let $\chi$ be a non-trivial character of $T$ and let
$E = \bigoplus \left(E_q \otimes L_{\chi^q}\right)$, where each $E_q$ is either 
zero or a direct sum of line bundles on $X$. Assume that 
$d = {\rm rank}(E) \ge 1$. Then $c^T_d(E)$ is of the form
\[
c^T_d(E) = u\left(x^d + \gamma_{d-1}x^{d-1} + \cdots + \gamma_1 x + \gamma_0\right),
\]
with $x = c^T_1\left(L_{\chi}\right)$ such that \\
$(i)$ \ $u$ is an invertible and homogeneous element of degree zero in
$\Omega^*_T(X)$, and \\
$(ii)$ \ $\gamma_i = u_i\gamma'_i$, where $u_i$ is an invertible and 
homogeneous element of degree zero in $\Omega^*_T(X)$ and $\gamma'_i$
is a product of the first Chern classes of line bundles on $X$. \\
In particular, each $\gamma_i$  is a nilpotent homogeneous element of 
$\Omega^*_T(X)$.
\end{lem}
\begin{proof}
By Lemma~\ref{lem:trivial-T}, $\Omega^*_T(X)$ is a graded power series
ring $\Omega^*(X)[[t_1, \cdots , t_n]]$.
Setting $E_q = \stackrel{d_q}{\underset{i =1}\oplus} L_{q,i}$, 
$v_{q,i} = c_1(L_{q,i})$ and 
$x_q = c^T_1\left(L_{\chi^q}\right)$, we can write 
\begin{equation}\label{eqn:CP0}
c^T_{d_q}\left(E_q \otimes L_{\chi^q}\right)
= \stackrel{d_q}{\underset{i=1}\prod} F(x_q, v_{q,i})
= \stackrel{s}{\underset{i=1}\prod} u_{q,i}(x_q + v_{q,i}),
\end{equation}
where $u_{q,i}$ is an invertible homogeneous element of degree zero in 
$\Omega^*_T(X)$.
Using the formula $x_q = [q]_Fx = qx + x^2 {\underset{i \ge 0}\sum} a_i x^i$, 
we can write $x_q = qx(1 + xg_q(x))$, where $1+ xg_q(x)$ is a homogeneous 
element of degree zero in $S(T)$ which is invertible. Setting 
$u_q = q(1+ xg_q(x))$, we obtain
\begin{equation}\label{eqn:CP1}
c^T_{d_q}\left(E_q \otimes L_{\chi^q}\right) =
\stackrel{d_q}{\underset{i=1}\prod} u_{q,i}(u_qx + v_{q,i}).
\end{equation}
The desired form of $c^T_{d}(E)$ follows immediately from this and
the  Whitney sum formula.
\end{proof}

\begin{cor}\label{cor:rigiditysuff}
Let $T$ be a split torus of rank $n$ acting trivially on a smooth variety $X$ 
and let $E$ be a $T$-equivariant vector bundle of rank $d$ on $X$. Assume that
in the eigenspace decomposition of $E$ with respect to 
$T$, the submodule corresponding to the trivial character is zero. Then
$c^T_d(E)$ is a non-zero divisor in $\Omega^*_T(X)$ and is invertible in 
$\Omega^*_T[M^{-1}]$.
\end{cor}
\begin{proof} 
This follows immediately from Lemmas~\ref{lem:SPLIT} and ~\ref{lem:RIGID*}.
\end{proof}

\section{Localization theorems}\label{section:LOCMAIN}
We now prove our following localization theorems for the equivariant
cobordism of smooth filtrable varieties with torus action.

\begin{thm}\label{thm:Loc-Main}
Let $X$ be a smooth and filtrable variety with the action
of a split torus $T$ of rank $n$. 
Let $i : X^T \inj X$ be the inclusion of the fixed point locus. Then the 
$S$-algebra map 
\[
i^* : \Omega^*_T(X) \to \Omega^*_T(X^T)
\]
is injective and is an isomorphism over $S[M^{-1}]$. 
\end{thm}
\begin{proof}
Consider the filtration of $X$ as in ~\eqref{eqn:filtration-BB}. 
It follows from the description of this
filtration that the tangent space $T_x(Z_0)$ is the weight zero
subspace ${T_x(X)}_0$ of the tangent space $T_x(X)$. Moreover, the
bundle $X_0 \to Z_0$ corresponds to the positive weight space
${T_x(X)}_{+}$ ({\sl cf.} \cite[Theorem~3.1]{Brion2}). 
It follows from this that in the 
eigenspace decomposition of the $T$-equivariant vector bundle $\phi_0 :
X_0 \to Z_0$, the submodule corresponding to the trivial character is zero.
We conclude from Corollary~\ref{cor:rigiditysuff} that the top Chern class
of the normal bundle of $Z_0$ in $X$ a non-zero divisor in $\Omega^*_T(Z_0)$. 
Since $X_0 \xrightarrow{\phi_0} Z_0$ is a $T$-equivariant vector bundle,
it follows from the homotopy invariance that 
$c_0 : = c^T_{d_0}\left(N_{{X_0}/X})\right)$ is a non-zero divisor in 
$\Omega^*_T(X_0)$. In particular, we conclude from 
Proposition~\ref{prop:split-cob} that the sequence 
\begin{equation}\label{eqn:Loc-Main1}
0 \to \Omega^*_T(X_0) \xrightarrow{i_*} \Omega^*_T(X) \xrightarrow{j^*}
\Omega^*_T(U_0) \to 0
\end{equation}
is exact and there is a ring isomorphism
\begin{equation}\label{eqn:Loc-Main2}
\Omega^*_T(X) \xrightarrow{\cong} \Omega^*_T(X_0) 
{\underset {\wt{\Omega^*_T(X_0)}}
{\times}} \Omega^*_T(U_0),
\end{equation}
where $\wt{\Omega^*_T(X_0)} = {\Omega^*_T(X_0)}/{(c_0)}$. In other words, there
is a short exact sequence 
\begin{equation}\label{eqn:Loc-Main3}
0 \to \Omega^*_T(X) \xrightarrow{(i^*, j^*)} \Omega^*_T(X_0) \times
\Omega^*_T(U_0) \to \frac{\Omega^*_T(X_0)}{(c_0)} \to 0.
\end{equation}
Furthermore, $c_0$ is invertible in $\Omega^*_T(X_0)[M^{-1}]$ by 
Corollary~\ref{cor:rigiditysuff}.
The proof of the theorem follows immediately from ~\eqref{eqn:Loc-Main3} and 
an induction argument once we observe that $U_0$ is itself a smooth and
filtrable variety with smaller number of strata than in $X$. 
\end{proof}

\begin{cor}\label{cor:Loc-MainP}
Let $X$ be a smooth and projective variety with the action
of a split torus $T$ of rank $n$. 
Let $i : X^T \inj X$ be the inclusion of the fixed point locus. Then the 
$S$-algebra map 
\[
i^* : \Omega^*_T(X) \to \Omega^*_T(X^T)
\]
is injective and is an isomorphism over $S[M^{-1}]$.
\end{cor} 

\begin{proof}
This follows immediately from Theorems~\ref{thm:BBH} and
~\ref{thm:Loc-Main}.
\end{proof}

\begin{cor}[Localization formula]\label{cor:Main-Loc-PF}
Let $X$ be a smooth and projective variety with an action of a split torus $T$
of rank $n$. Then the push-forward map
\[
\Omega^T_*(X^T) \xrightarrow{i^T_*} \Omega^T_*(X)
\]
is an isomorphism of $S$-modules over $S[M^{-1}]$.

In particular, one has for any $\alpha \in 
\Omega^*_T(X)[M^{-1}]$,
\[
\alpha = {\underset{F}\sum} {i_F}_* \frac{i^*_F(\alpha)}
{c^T_{d_F}(F_FX)},
\]
where the sum is over the components $F$ of $X^T$ and $d_F$ is the codimension
of $F$ in $X$.
\end{cor}
\begin{proof} This follows immediately from Corollary~\ref{cor:Loc-MainP}
and Proposition~\ref{prop:SIF}. 
\end{proof}
\begin{remk}\label{remk:EGB}
The analogous result for the $T$-equivariant Chow groups was proven
by Edidin and Graham \cite{EG1} and also by Brion \cite{Brion2}.
Edidin and Graham prove their result for Chow groups also for singular
varieties by a different method. 
\end{remk}
Let $E_1, \cdots , E_s$ be a set of $T$-equivariant vector bundles
on an $n$-dimensional smooth projective variety $X$.                
Let $p(x^1_1 , \cdots , x^1_s , \cdots , x^n_1, \cdots , x^n_s)$ 
be a polynomial of weighted degree $n$, where $x^i_j$ has weighted degree $i$.
Let $p(E_1 , \cdots , E_s)$ denote the polynomial in the Chern classes of 
$E_1 , \cdots , E_s$  obtained by setting $x_j^i = c_i(E_j)$. 
Let $q_X : X \to {\rm Spec}(k)$ be the structure map. Since $X$ is projective,
there is a push-forward map ${\rm deg} = {q_X}_* : \Omega_0(X) \to 
\Omega_0(k) \cong \Q$. As an immediate consequence of the above
localization formula, we get the following {\sl Bott residue formula}
for the algebraic cobordism which computes the degree of the cobordism cycle
${\rm deg}\left(p(E_1 , \cdots , E_s) \cap [X]\right)$ in terms of the
restriction of $E_i$ on $X^T$.
\begin{cor}[Bott residue for cobordism]\label{cor:BRF} 
Let $E_1, \cdots ,, E_s$ be a set of $T$-equivariant vector bundles
on an $n$-dimensional smooth projective variety $X$ as above. Then
\[
{\rm deg}\left(p(E_1 , \cdots , E_s) \cap [X]\right) =
{\underset{F}\sum}{q_F}_*\left(\frac{p^T(E|_F) \cap [F]_T}
{c^T_{d_F}(N_FX)}\right).
\]
\end{cor}
\begin{proof}
This is a direct consequence of Corollary~\ref{cor:Main-Loc-PF} and can be
proved exactly in the same way as the proof of the analogous result for
Chow groups in \cite[Theorem~3]{EG1}. We skip the details.
\end{proof}

\subsection{Description of the image of $i^*$}
The following result for equivariant Chow groups was proven by
Brion in \cite[Theorem~3.3]{Brion2}.
\begin{thm}\label{thm:REST}
Let $X$ be a smooth filtrable variety where a split torus $T$ of
rank $n$ acts with finitely many isolated fixed points. Let $i : X^T \to X$ be 
the inclusion of the fixed point locus.
Then the image of $i^*: \Omega^*_T(X) \to \Omega^*_T(X^T)$ is the intersection
of the images of the restriction maps
\[
i^*_{T'}: \Omega^*_T(X^{T'}) \to \Omega^*_T(X^T)
\]
where $T'$ runs over all subtori of codimension one in $T$. 
\end{thm}
\begin{proof} 
This follows from the results of the previous sections 
and by following the strategy of Brion. We give the main steps.
We prove by induction over the number of strata. If $X$ is the unique
stratum, it is a $T$-equivariant vector bundle over $X^T$, in which case 
both the maps $i^*$ and $i^*_{T'}$ are surjective by homotopy invariance.
In the general case, let $X_0 \subset X$ be a closed stratum and let
$U_0$ be its complement. Then $U_0$ is a filtrable smooth variety with
smaller number of strata where $T$ acts with isolated fixed points. We have 
seen in the proof of Theorem~\ref{thm:Loc-Main} that there are exact sequences
\begin{equation}\label{eqn:Loc-Main3*}
0 \to \Omega^*_T(X_0) \xrightarrow{{i_0}_*} \Omega^*_T(X) 
\xrightarrow{j^*_0} \Omega^*_T(U_0) \to 0
\end{equation}
\begin{equation}\label{eqn:Loc-Main3*1}
0 \to \Omega^*_T(X) \xrightarrow{(i^*_0, j^*_0)} \Omega^*_T(X_0) \times
\Omega^*_T(U_0) \to \frac{\Omega^*_T(X_0)}{(c_0)} \to 0
\end{equation}
where $c_0 = c^T_{d_0}\left(N_{{X_0}/{X}}\right)$ is the top Chern class of the
normal bundle of $X_0$ in $X$. Moreover, we have $i^*_0 \circ {i_0}_* =
c_0$. 

We identify $\Omega^*_T(X_0)$ with
$\Omega^*_T(Z_0)$ which in turn is identified with $\Omega^*(Z_0)[[t_1, 
\cdots , t_n]]$ by Lemma~\ref{lem:trivial-T}. In particular, we can
evaluate $c_0$ by identifying its class in $\Omega^*_T(Z_0)$ under the 
pull-back via the zero-section. Under this identification, we can write
$N_{{X_0}/{X}}$ as 
\[
N_{{X_0}/{X}} = \stackrel{s}{\underset{j = 1}\oplus} E_j
\]
where each $E_j$ is obtained by grouping together $E_{\chi}$ and $E_{\chi'}$ when 
the characters $\chi$ and $\chi'$ are multiples of a common primitive
character of $T$. In particular, we can write $E_j =
\oplus (E_q \otimes L_{\chi_j^q})$, where $\chi_j$ is a primitive character.
We then have 
\begin{equation}\label{eqn:Loc-Main3*2}
c_0 = \stackrel{s}{\underset{j = 1}\prod} c^T_{d_j}(E_j) 
= \stackrel{s}{\underset{j = 1}\prod}c_{\chi_j}.
\end{equation}

Using ~\eqref{eqn:Loc-Main3*},  ~\eqref{eqn:Loc-Main3*1} and
~\eqref{eqn:Loc-Main3*2}
and following the argument in \cite[Theorem~3.3]{Brion2}, we only need to show
that  
\begin{equation}\label{eqn:Loc-Main3*4}
(c_0) = \ \stackrel{s}{\underset{j = 1}\bigcap} \left(c^T_{d_j}(E_j)\right)
\ {\rm as \ ideals \ in} \ \Omega^*(Z_0)[[t_1, \cdots, t_n]]. 
\end{equation}

Since all vector bundles on $Z_0$ are trivial, it follows from 
Lemma~\ref{lem:CP} that $c^T_{d_j}(E_j) = u_j\left(c^T_1(L_{\chi_j})\right)^{d_j}$
where $u_j$ is invertible in $\Omega^*_T(Z_0) \cong S$.
Since $\{\chi_1, \cdots, \chi_s\}$ is pairwise 
non-proportional and since we are working with the rational coefficients,
setting $\gamma_j = \left(c^T_1(L_{\chi_j})\right)^{d_j}$, it follows from
Lemma~\ref{lem:elem1} that
\[
\left({\gamma}_1 \cdots {\gamma}_s\right) \ = \
\stackrel{s}{\underset{j=1}{\bigcap}} \left({\gamma}_j\right).
\]
Since each $u_j$ is a unit in $\Omega^*_T(Z_0)$, the equality of 
~\eqref{eqn:Loc-Main3*4} now follows. This completes the proof of the theorem.
\end{proof}

\begin{remk}\label{remk:RESTChow}
The analogue of Theorem~\ref{thm:REST} for the Chow groups was proven by 
Brion without any condition on the nature of the fixed point locus.
At this moment, the author is not confident that the result for cobordism
might hold in such generality because of the complicated nature of the
associated formal group law. However, the above theorem covers most of the
situations that occur in practice. In particular, this is applicable
for toric varieties, all flag varieties and all symmetric varieties of
minimal rank. These are the cases which have been of main interest of
enumerative geometry in recent past.
\end{remk}

The analogue of the following result for the equivariant Chow groups
was proven by Brion in \cite[Theorem~3.4]{Brion2}.
\begin{thm}\label{thm:LOC-APP}
Let $X$ be a smooth filtrable variety where a split torus $T$ acts with
finitely many fixed points $x_1, \cdots x_m$ and with finitely many invariant 
curves. Then the image of
\[
i^* : \Omega^*_T(X) \to \Omega^*_T(X^T)
\]
is the set of all $(f_1, \cdots , f_m) \in {S}^m$ such that
$f_i \equiv f_j$ (mod $\chi$) whenever $x_i$ and $x_j$ lie in an invariant
irreducible curve $C$ and the kernel of the $T$-action on $C$ is the kernel
of the character $\chi$.
\end{thm}
\begin{proof}
This is an easy consequence of Theorem~\ref{thm:REST}, as shown by Brion
for the equivariant Chow groups. Notice that $ \Omega^*_T(X^T)$ is
identified with ${S}^m$ by Lemma~\ref{lem:trivial-T}.

Let $\pi$ be a non-trivial character of $T$. Then the space $X^{{\rm Ker}(\pi)}$
is at most one-dimensional by our assumption, and is smooth.

In particular, it is a disjoint union of points and smooth connected curves.
If $C$ is such a curve and contains a unique fixed point $x$, then 
$i^*_x : \Omega^*_T(C) \to \Omega^*_T(\{x\}) = S$ is   
an isomorphism. Otherwise, $C$ must be isomorphic to the projective line
with fixed points $x$ and $y$. It follows from Example~\ref{exm:proj-space}
(see the exact sequence ~\eqref{eqn:SReis0}) that the image of 
\[
i^*_C : \Omega^*_T(C) \to \Omega^*_T(C^T)
\]
is of the desired form. Since
every codimension one subtorus in $T$ is the kernel of a non-trivial
character, the main result now follows from Theorem~\ref{thm:REST}.
\end{proof}

\begin{remk}\label{remk:BRFChow}
We remark here that the localization theorems and the Bott residue formula
for the equivariant Chow groups can be recovered from the above
results for the equivariant cobordism and \cite[Proposition~7.1]{Krishna4}.
\end{remk}

\section{Cobordism ring of flag varieties}\label{section:FLAG}
Let $G$ be a connected reductive group with a split maximal torus $T$ 
of rank $n$ and let
$B$ be a Borel subgroup of $G$ containing $T$. Let $X = G/B$ be the associated
flag variety of the left cosets of $B$ in $G$. We have shown in
Corollary~\ref{cor:FLAG-BASIS} that $\Omega^*_T(X)$ is a free $S(T)$-module
on the classes of the Bott-Samelson varieties. The similar conclusion
holds for the non-equivariant cobordism of $X$ as well. The Schubert 
calculus for the non-equivariant algebraic cobordism has been studied recently 
by Calmes, Petrov and Zainoulline in \cite{CPZ} (see also \cite{HK}). 
Using this, Calmes-Petrov-Zainoulline have obtained a description of
the cobordism ring $\Omega^*(X)$.

Theorem~\ref{thm:flag-V} below can be viewed as the
uncompleted version of the result
of \cite{CPZ} and gives a more geometric expression of $\Omega^*(X)$
in terms of the cobordism ring of the classifying space of the maximal torus 
$T$. This formula is a direct analogue of a similar result  of Demazure
\cite{Demazure} for the Chow ring of $X$.
Theorem~\ref{thm:flag-V} is used in \cite{Krishna3} to describe the algebraic 
cobordism of flag bundles. 

Consider the forgetful map $\Omega^*_G(X) 
\xrightarrow{r^G_X} \Omega^*(X)$. Using Proposition~\ref{prop:Morita},
we can identify $\Omega^*_G(X) = \Omega^*_G(G/B)$ with $\Omega^*_B(k)
\cong S(T)$. Thus $r^G_X$ is same as the map $S(T) \xrightarrow{r^G_X}  
\Omega^*(X)$.

\begin{thm}\label{thm:flag-V}
The ring homomorphism $r^G_X$ descends to an isomorphism of
$\bL$-algebras
\begin{equation}\label{eqn:SII*}
c_X : S(T) \otimes_{S(G)} \bL \to \Omega^*(X).
\end{equation}
In particular, if $G = GL_{n}(k)$, then $\Omega^*(G/B)$ is isomorphic to the 
quotient of the standard polynomial ring $\bL[t_1, \cdots , t_n]$ by the 
ideal $I$ generated by the homogeneous symmetric polynomials of strictly 
positive degree.
\end{thm}
\begin{proof}
We first prove the isomorphism of $c_X$.
The forgetful map $S(T) = \Omega^*_G(X) \to \Omega^*(X)$ can be
geometrically described as follows. Let $L_{\chi}$ be the $T$-equivariant
line bundle on ${\rm Spec}(k)$ corresponding to the character $\chi$ of $T$.
This uniquely defines a line bundle $L_{\chi} \stackrel{B}{\times} G \to G/B = X$
on $X$. We denote this line bundle by $L_{\chi, X}$. The formal group law
for the Chern classes in $S(T)$ and $\Omega^*(X)$ implies that the
assignment $\chi \mapsto L_{\chi, X}$ induces an $\bL$-algebra homomorphism
\begin{equation}\label{eqn:SII*1}
S(T) \to \Omega^*(X), \ c^T_1(L_{\chi}) \mapsto c_1\left(L_{\chi, X}\right)
\end{equation}
and it is easy to see from the definition of $r^G_X$ in ~\eqref{eqn:res}
and the identification $\Omega^*_G(X) \xrightarrow{\cong} S(T)$ that
this map descends to the map $c_X$ above.

We identify $S(T)$ with the graded power series ring 
$\bL[[t_1, \cdots , t_n]]$ and let $I_T$ be the augmentation ideal 
$(t_1, \cdots , t_n)$. Taking the $I_T$-adic completions, we get the map
$\widehat{S(T)} = \widehat{\bL[[{\bf t}]]} \xrightarrow{\widehat{c}_X} 
\widehat{\Omega^*(X)}$.

We claim that $c_X$ is surjective. To prove this, let $N$ denote
the cokernel of the map $S(T) \to \Omega^*(X)$ and consider the
exact sequence of $S(T)$-modules
\begin{equation}\label{eqn:SII*2}
0 \to M \to \Omega^*(X) \to N \to 0,
\end{equation}
where $M = {\rm Image}\left(S(T) \to \Omega^*(X)\right)$.
Using the subspace topology on $M$ given by the descending chain of
submodules $\{I^n_T\Omega^*(X) \cap M\}_{n \ge 0}$, we get an exact 
sequence of completions
\begin{equation}\label{eqn:SII*3} 
0 \to \widehat{M} \to \widehat{\Omega^*(X)} \to \widehat{N} \to 0
\end{equation}
by \cite[Theorem~8.1]{Matsumura}.

Since $c_X(t_i) = c_1(L_{t_i, X}) \in \Omega^1(X)$ and since 
$\Omega^{> {\rm dim}(X)}(X) = 0$,
we see that $c_X(t_i)$ is nilpotent for all $i$. In particular,
$\Omega^*(X)$ is $I_T$-adically complete and hence so is $N$.
In particular, the exact sequence ~\eqref{eqn:SII*3} is same as
\begin{equation}\label{eqn:SII*4} 
0 \to \widehat{M} \to {\Omega^*(X)} \to N \to 0.
\end{equation}
On the other hand, the map $\widehat{S(T)} \to \Omega^*(X)$ is surjective by
\cite[Corollary~13.6]{CPZ}. Since the $I_T$-adic topology on $M$ is finer
than the subspace topology, we have natural maps
$\widehat{S(T)} \to \widehat{M}_{I_S} \to \widehat{M} \to \Omega^*(X) = 
\widehat{\Omega^*(X)}$.
We conclude that the first arrow in ~\eqref{eqn:SII*4} is surjective and
hence $N = 0$, which proves the surjectivity of $c_X$.

Since $S(T)$ is a subring of $\widehat{S(T)}$, it follows from 
\cite[Corollary~13.7]{CPZ} that the kernel of the map
$S(T) \to \Omega^*(X)$ is generated by $(I_T)^W$, where $W$ is the Weyl group
of $G$ with respect to $T$.
On the other hand, the commutative diagram of exact sequences
\[
\xymatrix@C.5pc{
0 \ar[r] &I_G \ar[r] \ar[d] & S(G) \ar[r] \ar[d] & \bL \ar[r] \ar@{=}[d] & 
0 \\
0 \ar[r] & I_T \ar[r] & S(T) \ar[r] & \bL \ar[r] & 0,}
\]
the exactness of the functor of $W$-invariance
on the category of $\Q[W]$-modules and Theorem~\ref{thm:W-inv} together
imply that $I_G = (I_T)^W$.
In particular, we get exact sequence
\[
I_G \otimes_{S(G)} S(T) \to S(T) \to \Omega^*(X) \to 0
\]
which completes the proof of the main part of the theorem.

If $G = GL_n(k)$, then we can choose $T$ to be the set of all diagonal matrices
which is clearly split and $W$ is just the symmetric group $S_n$ which acts
on $S(T) = \bL[[t_1, \cdots , t_n]]$ by permuting the variables. In
particular, $S(G)$ is the graded power series subring of 
$S$ generated by the elementary symmetric polynomials. It is easy to see from
this that $S \otimes_{S(G)} \bL$ has the desired form.
\end{proof}

\begin{remk}
For $G = GL_n(k)$, the cobordism ring of $G/B$ in the above explicit form has 
been recently computed by Hornbostel and Kiritchenko \cite{HK} with integer 
coefficients. They achieve this using the approach of Schubert calculus.  
\end{remk}

As an easy consequence of Theorems~\ref{thm:Basic},  
~\ref{thm:FF*} and ~\ref{thm:flag-V}, one gets the following description of 
the rational cobordism ring of connected reductive groups.

\begin{cor}\label{cor:Cob-red}
For a connected and split linear algebraic group $G$, the natural ring 
homomorphism
\[
\bL \to \Omega^*(G)
\]
is an isomorphism.
\end{cor}
\begin{proof}
Let $G^u$ be the unipotent radical of $G$ and let $L \subset G$ be a Levi
subgroup (which exists because we are in characteristic zero).
Since $G \to L$ is principal $G^u$ bundle and since $G^u$ is a split 
unipotent group (again because of characteristic zero), the homotopy 
invariance implies that $\Omega^*(L) \xrightarrow{\cong} \Omega^*(G)$.
So we can assume that $G$ is a connected and split reductive group.
  
We choose a split maximal torus $T$ of $G$ of rank $n$ and let 
$\{\chi_1, \cdots , \chi_n\}$ be a basis of the character group of $T$. 
By Theorem~\ref{thm:FF*}, there is an isomorphism of rings
$\Omega^*_T(G) \otimes_{S(T)} \bL \xrightarrow{\cong} \Omega^*(G)$. On the other
hand, it follows from Theorem~\ref{thm:Basic} $(v)$ that
$\Omega^*_T(G) \cong \Omega^*(G/T) \cong \Omega^*(G/B)$. 
In particular, we get 
\[
\begin{array}{lllll}
\Omega^*(G) & \cong & \frac{\Omega^*_T(G)}{\left(t_1, \cdots , t_n\right)} &
\cong & \frac{\Omega^*(G/B)}{\left(t_1, \cdots , t_n\right)} \\
& & & \cong & 
\frac{S(T) \otimes_{S(G)} \bL}{\left(t_1, \cdots , t_n\right)} \\
& & & \cong & \bL,
\end{array}
\]
where  the third isomorphism follows from Theorem~\ref{thm:flag-V}.
\end{proof}

\begin{remk}\label{remk:Yagita}
For a complex connected Lie group $G$ such that its maximal compact subgroup
is simply connected, $\Omega^*(G)$ has been computed
with integer coefficients by Yagita \cite{Yagita} using non-equivariant
techniques.  
\end{remk}

It is known ({\sl cf.} \cite[Theorem~1.1]{Graham}, 
\cite[Corollary~2.3]{Brion2}) that the forgetful map from the equivariant
$K_0$ (resp. Chow groups) to the ordinary $K_0$ (resp. Chow groups) is
surjective with the rational coefficients. As  another application of above 
results, we obtain a similar result for the cobordism.

\begin{thm}\label{thm:forget}
Let $G$ be a connected linear algebraic group acting on a $k$-scheme $X$
of dimension $d$.
Then the forgetful map $r^G_X : \Omega^G_*(X) \to \Omega_*(X)$ descends to 
an $\bL$-linear map $\Omega^G_*(X) \otimes_{S(G)} \bL \to \Omega_*(X)$ and
is surjective with rational coefficients.
\end{thm}
\begin{proof}
By \cite[Proposition~8.2]{Krishna4}, we can assume that $G$ is reductive with a 
maximal torus $T$. There is a finite field 
extension $k \inj l$ such that $T_l$ is split. 
We then have a commutative diagram
\begin{equation}\label{eqn:W-inv1*1}
\xymatrix@C1.8pc{
\Omega^{G_l}_*(X_l) \ar[r]^{\pi^G_*} \ar[d] & \Omega^G_*(X) \ar[d] \\ 
\Omega_*(X_l) \ar[r]_{\pi_*} & \Omega_*(X).}
\end{equation}
The map $\pi_* \circ \pi^*$ is multiplication by $[l:k]$ by 
\cite[Lemma~2.3.5]{LM}. In particular, the bottom horizontal map is
surjective with rational coefficients. Hence, we can assume that $T$ is 
a split torus. Let $W$ denote the Weyl group of $G$ with respect to $T$.
The commutative diagram
\[
\xymatrix@C.7pc{
I_G \otimes_{S(G)} \Omega^G_*(X) \ar[r] \ar[d] &
\Omega^G_*(X)  \ar[r] \ar[d] & \Omega_*(X) \ar[d] \\ 
I_T \otimes_{S(T)} \Omega^T_*(X) \ar[r] &  \Omega^T_*(X) \ar[r] & \Omega_*(X)}
\]
and Theorem~\ref{thm:FF*} imply that $r^G_X$ descends to the map
$\Omega^G_*(X) \otimes_{S(G)} \bL \xrightarrow{\ov{r}^G_X} \Omega_*(X)$
which is a ring homomorphism if $X$ is smooth.

To show its surjectivity, we use Theorem~\ref{thm:W-inv} to get
\[
\Omega^G_*(X)  \otimes_{S(G)} \bL \cong 
\left(\Omega^T_*(X)\right)^W \otimes_{S(G)} \bL 
\cong \left(\Omega^T_*(X)  \otimes_{S(G)} \bL\right)^W.
\]
On the other hand, we have
\[
\Omega^T_*(X)  \otimes_{S(G)} \bL \cong
\Omega^T_*(X) \otimes_{S(T)} \left(S(T) \otimes_{S(G)} \bL\right)
\surj \Omega^T_*(X) \otimes_{S(T)} \bL \cong \Omega_*(X),
\]
where the last isomorphism follows from Theorem~\ref{thm:FF*}.
The exactness of the functor of taking the $W$-invariance on the category
of $\Q[W]$-modules implies that
$\left(\Omega^T_*(X)  \otimes_{S(G)} \bL\right)^W \surj \left(\Omega_*(X)\right)^W
= \Omega_*(X)$. This proves the surjectivity
$\Omega^G_*(X) \surj \Omega^G_*(X)  \otimes_{S(G)} \bL \surj \Omega_*(X)$.
\end{proof}

\noindent\emph{Acknowledgments.}
Some of the results in Section~\ref{section:LOCMAIN} of this paper were 
inspired by the discussion with Michel Brion at the Institute Fourier, 
Universit\'e de Grenoble in June, 2010. The author takes this opportunity to 
thank Brion for invitation and financial support during the visit.

\end{document}